\documentclass[12pt, reqno]{amsart}
\usepackage{amsfonts,amsmath,amssymb,amsthm,epsfig,cite,graphicx,hyperref,
	color, esint,fancyhdr, enumerate, latexsym,amsrefs, makecell}
\usepackage{color}
\usepackage[mathscr]{eucal}
\usepackage{comment}
\usepackage{xcolor}
\usepackage{tikz}
\usetikzlibrary{arrows.meta,calc}
\usepackage{float}

\usepackage[nameinlink]{cleveref}

\numberwithin{equation}{section}

\theoremstyle{definition}
\newtheorem{definition}{Definition}[section]
\newtheorem{example}[definition]{Example}

\theoremstyle{remark}
\newtheorem{remark}[definition]{Remark}
\theoremstyle{plain}
\newtheorem{theorem}[definition]{Theorem}
\newtheorem{lemma}[definition]{Lemma}
\newtheorem{proposition}[definition]{Proposition}

\newtheorem{result}[definition]{Result}

\newtheorem{corollary}[definition]{Corollary}

\newtheorem*{claim*}{Claim}

\newcommand{\eps}{\varepsilon}

\newcommand{\Om}{\Omega}
\newcommand{\mcs}{\mathcal{S}}

\newcommand{\st}{\subset}
\newcommand{\St}{\Subset}

\newcommand{\mco}{\mathcal{O}}
\newcommand{\D}{\mathbb{D}}
\newcommand{\bn}{\mathbb{B}^{d}}

\newcommand{\kb}{K_{\mathbb{B}^{d}}}

\newcommand{\mcd}{\mathcal{D}}

\newcommand{\smoo}{\mathcal{C}}

\newcommand{\rl}{{\sf Re}}

\newcommand{\impl}{\Longrightarrow}

\newcommand{\cplx}{\mathbb{C}}

\newcommand{\re}{\mathbb{R}}
\newcommand{\cn}{\mathbb{C}^d}

\begin{document}
	\title[On exponential convergence of Kobayashi geodesics]{On the exponential convergence of Kobayashi geodesics in  strongly convex domains}
	\author{Kingshook Biswas,  Sanjoy Chatterjee and Amar Deep Sarkar}
   \address{Stat-Math Unit, Indian Statistical Institute, 203 B. T. Rd., Kolkata 700108, India}
   \email{kingshook@isical.ac.in}
	\address{Department of Mathematics and Statistics, Indian Institute of Technology,
				KANPUR}
		\email{ramvivsar@gmail.com }
       \address{Department of Mathematics
IIT Bhubaneswar}
		\email{amar@iitbbs.ac.in }
        

	\date{\today}
\begin{abstract}
In this paper, we have proved a quantitative version of the approaching   geodesic property for certain convex domains.  We have proved that 
that if $\Om \subset \cn$ is a bounded strongly convex domain with $\smoo^3$ boundary and $\gamma_{1}, \gamma_{2}:[0, \infty) \to \Om$ are two geodesic rays such that $\gamma_{1}(\infty)=\gamma_{2}(\infty)=\xi \in \partial \Om$.  Then  if the images of $\gamma_{1}$ and $\gamma_{2}$ are contained in the same complex geodesic, then there exists $T\in \mathbb{R}$
\[
\lim_{t \to \infty} \frac{1}{t}
\log K_{\Om}\big(\gamma_{1}(t), \gamma_{2}(t+T)\big)
= -2,
\]
otherwise
\[
\lim_{t \to \infty} \frac{1}{t}
\log K_{\Om}\big(\gamma_{1}(t), \gamma_{2}(t+T)\big)
= -1.
\]    
We prove that if a bounded strongly convex domain with smooth boundary, endowed with its Kobayashi distance, is a ${\rm CAT}(-1)$ space, then it is biholomorphic to the unit ball. Consequently, although a general bounded strongly convex domain need not be ${\rm CAT}(-1)$, the preceding theorem shows that such domains enjoy a key  property of  of ${\rm CAT}(-1)$ spaces.
Furthermore, using this property we provided a  characterization of          strongly pseudoconvex domain via a biholomorphic invariant  function  namely {\em generalized squeezing function}. We have proved that: For every $\alpha>0$ there  exists $\eps(d,\alpha)>0$  such that the following holds: if \(\Omega \subset \mathbb{C}^d\) is a bounded convex domain with \(\smoo^{2,\alpha}\)-boundary and
\[
T_{\Omega}^{D}(z)\geq 1-\epsilon
\]
outside a compact subset of \(\Omega\),  where  $D \St \cn$ is a balanced  strongly convex domain with $\smoo^{3}$ boundary and $T_{\Om}^{D}$ is the squeezing function of $\Om$ with respect to the domain $D$  then \(\Omega\) is strongly pseudoconvex. We also establish exponential convergence of a certain family of quasi-geodesics in the unit ball of $\mathbb{C}^{d}$. We further show that the study of this family of quasi-geodesics provides a useful tool that allows the exponential convergence property of geodesics to be transferred from local subdomains to the ambient domain, as well as in the reverse direction.

\end{abstract}

\maketitle

\section{Introduction and statements of main results}

Let $\D=\{z\in \mathbb{C}: |z|<1\}$ and  $\rho_{\D}$ be the Poincar\'e distance on $\D$. The geometric properties of the Poincar\'e distance plays a central role in hyperbolic geometry and complex analysis. In several complex variables, a natural analogue of the Poincar\'e distance is the Kobayashi distance (see Definition~\ref{Def: Kobayashi hyperbolic}). The Kobayashi distance is a fundamental holomorphic invariant of a domain in higher dimensions, and its metric-geometric properties have proved to be a powerful tool in complex geometry. For example, the Gromov hyperbolicity and visibility  of the Kobayashi distance has been successfully exploited in the study of iteration theory and the boundary extension of holomorphic maps. Since Kobayashi hyperbolic domains provide the natural higher-dimensional analogue of the Poincaré ball model $(\mathbb{D},\rho_{\mathbb{D}})$, it is natural to investigate which geometric properties of the Poincaré ball persist in higher dimensions and what consequences arise from such properties. In \cite{Bracci2005}, Bracci and Patrizio studied a property of complex geodesics in strongly convex domains related to smooth foliation of strongly convex domains induced by complex geodesics, and obtaining necessary and sufficient criteria for a map between two strongly convex domains to be biholomorphic.    

\noindent
In this paper, we study one such property, namely the exponential convergence two geodesics landing at same boundary point of bounded strongly convex domains with $\smoo^3$ boundary. Let $\Omega \subset \mathbb{C}^d$ be a Kobayashi hyperbolic domain  and  $K_{\Omega}$ denote its Kobayashi distance. An important asymptotic property of geodesics in $(\Omega,K_{\Omega})$ is the \emph{approaching geodesic property}.  A complete hyperbolic domain $\Omega \st \cn$ is said to have the {\em approaching geodesic property}, if  whenever $\gamma_1,\gamma_2:[0,\infty)\to\Omega$ are geodesic rays  with  
\[
\sup_{t\ge 0} K_{\Omega}\bigl(\gamma_1(t),\gamma_2(t)\bigr)<\infty,
\]
there exists $T\in\mathbb{R}$ such that

\begin{align*}
  \lim_{t\to\infty}
K_{\Omega}\bigl(\gamma_1(t),\gamma_2(t+T)\bigr)=0. \end{align*}

\noindent
This property has important applications in the study of the relationship between the horofunction boundary and the Gromov boundary of proper metric spaces, as well as in the iteration theory of holomorphic maps (see \cite[Theorem~3.5]{AFG2024} and \cite{AF2024}). It is known that bounded strongly pseudoconvex domains with $\mathcal{C}^{2}$ boundary and bounded convex finite-type domains satisfy the approaching geodesic property (see \cite[Theorem~5.3]{AFG2024}).

\noindent
It is well known that for  $(\D, \rho_{\D})$ satisfies a considerably stronger property than the approaching geodesic property alone. In this case, geodesic rays with same landing point at the boundary of $\D$ converge to one another  ``exponentially fast".  More precisely, if $\gamma_1,\gamma_2:[0,\infty)\to\D$ are geodesic rays initiated on the same horosphere and satisfying
\[
\gamma_1(\infty)=\gamma_2(\infty)=\xi\in\partial\D,
\]
then there exists a constant $A>0$,  such that
\[
\rho_{\D}\bigl(\gamma_1(t),\gamma_2(t)\bigr)
\leq A e^{-t}
\qquad \text{for all } t\geq 0.
\]
\noindent
The above property is not specific to the Poincaré disc it also holds for certain general negatively curved Riemannian manifold. A synthetic framework of negatively curved manifold in which a similar phenomenon occurs is provided by \rm{CAT}(-1) spaces (see Definition~\ref{Def:cat}). Indeed, it is known that if two geodesic rays in a \rm{CAT}(-1) space determine the same point in the Gromov boundary, both are initiated from same horosphere, then the distance between them decays exponentially along the rays. In other words, asymptotic geodesic rays in a \rm{CAT}(-1) space approach one another exponentially fast (see \cite[Proposition~4.4.13]{DTD2017} for details). This property has played an important role in the study of the stability of the geodesic flows in ${\rm CAT}(-1)$ spaces (see, for instance, \cite{ConsLafonThomp2020a},\cite{ConsLafoThomp2020}). Since the unit ball $\mathbb{B}^{n}$ equipped with its Kobayashi distance is a ${\rm CAT}(-1)$ space, distance between any two asymptotic geodesic rays in $\mathbb{B}^{d}$ goes to zero exponentially faster. In \cite[Theorem~2.11]{Zimmer2018}, Zimmer provided a proof of  this fact. Let us mention it as a result. 

\begin{result}\label{R:expcball}
Suppose that  $\gamma_{1}, \gamma_{2} : [0,\infty)\to \mathbb{B}^{d}$ are two geodesic rays such that $\liminf_{s,t \to \infty} K_{\mathbb{B}^{d}}\big(\gamma_{1}(s), \gamma_{2}(t)\big) < \infty$. Then  if the images of $\gamma_{1}$ and $\gamma_{2}$ are contained in the same complex geodesic, then there exists $T\in \mathbb{R}$ such that
\[
\lim_{t \to \infty} \frac{1}{t}
\log K_{\mathbb{B}^{d}}\big(\gamma_{1}(t), \gamma_{2}(t+T)\big)
= -2,
\]
otherwise
\[
\lim_{t \to \infty} \frac{1}{t}
\log K_{\mathbb{B}^{d}}\big(\gamma_{1}(t), \gamma_{2}(t+T)\big)
= -1.
\]    
\end{result}
Therefore, Result~\ref{R:expcball} shows that, for the unit ball ,  with respect to the Kobayashi distance  not merely satisfies the {\em approaching geodesic property}; their asymptotic behavior can be described in a much more quantitative manner. 

In \cite{Zimmer2018}, the exponential convergence of asymptotic geodesics in the unit ball was crucially used to study the strong pseudoconvexity property of domains in $\cn$. Motivated by this, it is natural to ask which class of Kobayashi hyperbolic domains in $\mathbb{C}^d$ exhibits exponential convergence of approaching geodesics. The main goal of this article is to declare that a bounded strongly convex domain with $\smoo^3$ boundary in $\cn $ with respect the Kobayashi distance possess the ``exponential convergence'' property. More precisely our first main result is the following:

\begin{theorem}\label{T:expcon}
  Let $\Om \st \cn$ be a bounded strongly convex domain with $\smoo^{3}$ boundary. Suppose that  $\gamma_{1}, \gamma_{2} : [0,\infty)\to \Om$ are two geodesic rays such that $\gamma_{1}(\infty)=\gamma_{2}(\infty)$. Then  if the images of $\gamma_{1}$ and $\gamma_{2}$ are contained in the same complex geodesic, then there exists $T\in \mathbb{R}$
\[
\lim_{t \to \infty} \frac{1}{t}
\log K_{\Om}\big(\gamma_{1}(t), \gamma_{2}(t+T)\big)
= -2,
\]
otherwise
\[
\lim_{t \to \infty} \frac{1}{t}
\log K_{\Om}\big(\gamma_{1}(t), \gamma_{2}(t+T)\big)
= -1.
\]    
\end{theorem}

\medskip

\noindent
We  have proved in Proposition~\ref{Rem:CAT} that if $\Om \Subset \mathbb{C}^d$ is a strongly convex domain with $\mathcal{C}^{\infty}$-smooth boundary, then $(\Om,K_{\Om})$ is a ${\rm CAT}(-1)$ space if and only if $\Om$ is biholomorphic to $\bn$. Theorem~\ref{T:expcon} shows that bounded strongly convex domains with $\mathcal{C}^{\infty}$-smooth boundary enjoy the exponential convergence property, even though they are not necessarily ${\rm CAT}(-1)$ spaces.

\smallskip

\noindent
Next, we turn our attention to applications of Theorem~\ref{T:expcon}. Recall that a domain $\Omega \subset \mathbb{C}^d$ with $\mathcal{C}^2$ boundary is called strongly pseudoconvex if the Levi form associated with a defining function of $\Omega$ is positive definite on the complex tangent space at every boundary point. The study of analytic and geometric properties of strongly pseudoconvex domains are central theme of several complex variables. It is a question of interest  whether the strong pseudoconvexity can be detected through intrinsic invariants of the domain. In \cite{Zimmer2018}, Zimmer obtained a characterization of strong pseudoconvexity using intrinsic continuous functions. We now mention the definition of such functions from \cite{Zimmer2018}. Let $\mathbb{X}_d$ denotes the collection of all convex domains in $\mathbb{C}^d$ that do not contain any complex affine line, and $\mathbb{X}_{d,0}$ denotes the set of pairs $(\Omega,x)$ with $\Omega \in \mathbb{X}_d$ and $x \in \Omega$. 
\begin{definition}
   A function $f : \mathbb{X}_{d,0} \to \mathbb{R}$ is called is called \emph{intrinsic} if $f(C_1,x_1)=f(C_2,x_2)$ whenever there exists a biholomorphism $\varphi:C_1\to C_2$ such that $\varphi(x_1)=x_2.$
\end{definition}

An example of such intrinsic function is the generalized squeezing function (see Section~\ref{S:prelims} for its definition and properties)
\[
T_D:\mathbb{X}_{d,0}\to\mathbb{R},
\qquad
(\Omega,z)\mapsto T_D^{\Omega}(z),
\]
where $D\subset\mathbb{C}^d$ is a fixed bounded, balanced, convex domain. We refer to  \cite{Zimmer18Gap} for other examples of intrinsic  function.
Applying Result~\ref{R:expcball}, Zimmer  provided the following characterization of strong pseudoconvexity 
\begin{result}\label{R:Zimm}
Suppose that $f:\mathbb{X}_{d,0}\to\mathbb{R}$
is a continuous intrinsic function with the following property: if
$C\in\mathbb{X}_d$ and $
f(C,x)=f(\mathbb{B}^d,0)
$ for all $x\in C$, then $C$ is biholomorphic to $\mathbb{B}^d$. Then for every $\alpha>0$ there exists
$
\varepsilon=\varepsilon(d,f,\alpha)>0
$
such that the following holds. If $C\subset\mathbb{C}^d$ is a bounded convex domain with
$\mathcal{C}^{2,\alpha}$ boundary and
\[
\big|f(C,z)-f(\mathbb{B}^d,0)\big|\le \varepsilon
\]
outside some compact subset of $C$, then $C$ is strongly pseudoconvex. In particular,
\[
\lim_{z\to\partial C} f(C,z)=f(\mathbb{B}^d,0).
\]
\end{result}
Our next theorem establishes that, conclusion of the above result remains valid when the unit ball is replaced by an arbitrary fixed bounded balanced strongly convex domain with $\mathcal{C}^3$ boundary. More precisely, we have shown that
\begin{theorem}\label{T:Application}
Let $(\Om,0) \in \mathbb{X}_{d,0}$ and $\Om$ be a bounded strongly convex domain with $\smoo^{3}$  boundary. Suppose that $f : \mathbb{X}_{d,0} \to \mathbb{R}$ is a continuous intrinsic function with the following property: if $\mathcal{D} \in \mathbb{X}_d$ and
\[
f(\mathcal{D},x)=f(\Om,0)
\]
for all $x\in \mathcal{D}$, then $\mathcal{D}$ is biholomorphic to $\Om$. Then for any $\alpha>0$ there exists some $\epsilon=\epsilon(d,f,\alpha)>0$ such that: if $\mathcal{D}\subset \mathbb{C}^d$ is a bounded convex domain with $\smoo^{2,\alpha}$ boundary and
\[
\left|f(\mathcal{D},z)-f(\Om,0)\right|\leq \epsilon
\]
outside some compact subset of $\mathcal{D}$, then $\mathcal{D}$ is strongly pseudoconvex and thus
\[
\lim_{z\to \partial \mathcal{D}} f(\mathcal{D},z)=f(\Om,0).
\]
\end{theorem}

We also have the following  theorem which is similar to that of \cite[Theorem~5.6]{Zimmer2018}.

\begin{theorem}\label{T:app2}
Let $\Om \st \cn$ be any bounded, strongly convex domain with $\smoo^{3}$ boundary containing the origin. Suppose that $f : \mathbb{X}_{d,0} \to \mathbb{R}$ is an upper semi-continuous intrinsic function with the following property: $$  \mcd \in \mathbb{X}_d~~ {\rm and}~f(\mcd,x)\geq f(\Om,0)\quad
\forall x\in \mcd \impl \mcd \cong \Om.$$ Then for any \(\alpha>0\) there exists some $\epsilon=\epsilon(d,f,\alpha)>0$
such that the following holds: If \(\mcd\subset \mathbb{C}^d\) is a bounded convex domain with \(\smoo^{2,\alpha}\)-boundary and $$f(\mcd,z)\geq f(\Om,0)-\epsilon$$ outside some compact subset of \(\mcd\), then \(\mcd\) is strongly pseudoconvex.
\end{theorem}
\smallskip
\noindent
We observed in Lemma~\ref{L:gneralsqu} that the generalized squeezing function \(T^{D}_{\Omega}\), associated to a domain \(\Omega \subset \mathbb{C}^{d}\) with respect to a bounded, balanced,  convex domain \(D\), is an upper semicontinuous intrinsic function on \(\mathbb{X}_{d,0}\). Consequently, Theorem~\ref{T:app2} immediately yields the following characterization of strong pseudoconvexity for domains in \(\mathbb{C}^{d}\).
\begin{corollary}\label{Cor:cor1}
   Let $D \St \cn$ be a balanced, strongly convex domain with $\smoo^3$ boundary. Then for any \(d \geq 2\) and \(\alpha > 0\), there exists some $\epsilon=\epsilon(d,\alpha)>0$ such that the following holds: If \(\Omega \subset \mathbb{C}^d\) is a bounded convex domain with \(\smoo^{2,\alpha}\)-boundary and
\[
T_{\Omega}^{D}(z)\geq 1-\epsilon
\]
outside a compact subset of \(\Omega\), then \(\Omega\) is strongly pseudoconvex.
\end{corollary}
\noindent
It is worth mentioning that there exists a bounded, balanced, strongly convex domains with $\mathcal{C}^{\infty}$-smooth boundary that are not biholomorphic to the unit ball (see \cite[Example~1]{CG2025b}). Therefore, Theorem~\ref{T:Application} and Corollary~\ref{Cor:cor1} extend Result~\ref{R:Zimm} and \cite[Theorem~1.7]{Zimmer2018}.

\medskip

\noindent
We now discuss exponential convergence of two quasi-geodesics in $\bn$ landing at the same boundary point.  In a Gromov hyperbolic space, the Morse lemma asserts that every quasi-geodesic joining two points remains within a uniformly bounded Hausdorff distance of any geodesic connecting the same endpoints. This naturally raises the question of whether an analogous phenomenon holds asymptotically for quasi-geodesic rays landing at same end point. Our next result provides an affirmative answer for a certain class of quasi-geodesic rays in the unit ball of $\mathbb{C}^{d}$ by showing that they converge to each other at an exponential rate. We have proved that
\begin{theorem}\label{P:exponetial}
    Let $\gamma_{1}, \gamma_{2}:[0, \infty) \to \bn$ be two rays such that for all $s\geq 0$, $\gamma_{j}|_{[s, \infty)}$ is $(1, a e^{-b s})$ quasi-geodesics for $j=1,2$ for some $a, b>0$.  Assume that $\gamma_1(\infty) = \gamma_2(\infty) = \xi \in \partial \bn$. Then there exist $\alpha >0$ and $T_{0} \in \mathbb{R}$  such that 
    \begin{align*}
       K_{\bn}(\gamma_{1}(t), \gamma_{2}(t+T_{0}))\leq \alpha\cdot e^{-\frac{t}{2}\cdot\min\{1,  ~b/2\}} \hspace{1 em} \forall t \geq \max\{0,-T_{0}\}. 
    \end{align*}
    \end{theorem}
    
    \smallskip
    
\begin{remark}
 The key ingredient in the proof of Theorem~\ref{P:exponetial} is the fact that $(\bn,\kb)$ is a ${\rm CAT}(-1)$ space. In fact, the same conclusion remains valid for $(1,ae^{-bs})$ quasi-geodesic in any   ${\rm CAT}(-1)$  space satisfying the analogous hypothesis. 
\end{remark}
 We now highlight the significance of the exponential convergence of the special class of quasi-geodesics established in Theorem~\ref{P:exponetial}. In particular, the exponential convergence of the family of  quasi-geodesics in $\bn$  considered in Theorem~\ref{P:exponetial} provides a useful mechanism for constructing new subdomains of the ball that enjoy the exponential convergence property for geodesics terminating at certain  boundary points. Our next theorem  provides such a class of domains.
\begin{theorem}
    Let $\xi\in\partial\bn$ and let $U_{\xi}$ be an open neighbourhood of $\xi$ such that $B':=\bn\cap U_{\xi}$ is connected. Let $\gamma_1,\gamma_2:[0,\infty)\longrightarrow B'$ be two geodesic rays satisfying $\gamma_1(\infty)=\gamma_2(\infty)=\xi.$ Then there exist constants $a,b>0$ and $T\in\re$ such that
\begin{equation*}
K_{B'}\bigl(\gamma_1(t),\gamma_2(t+T)\bigr)
    \le a e^{-bt}
    \qquad t> \max\{0, -T\}.
\end{equation*}
\end{theorem}
The study of the convergence of $(1,ae^{-bs})$ quasi-geodesics locally near the boundary points of certain domains also allows us to deduce the exponential convergence property for geodesics in the ambient domain. More precisely, we prove the following local-to-global principle.

\begin{theorem}\label{T:local to global1}
 Let $\Omega \subset \mathbb{C}^{d}$ be a bounded, complete hyperbolic domain satisfying the geodesic visibility property. Assume that, for every $\xi \in \partial \Omega$, there exists an open neighborhood $U_{\xi}$ of $\xi$ such that $\Omega' := \Omega \cap U_{\xi}$ is connected and has the following property: for any two maps $\gamma_{1},\gamma_{2}:[0,\infty)\to\Omega'$ and  $a,b>0$, \[ \gamma_j\big|_{[s,t]} \text{ is a }(1,ae^{-bs})\text{-quasi-geodesic for every }0\le s<t, \quad j=1,2, \] implies that there exist $T_{0}\in\mathbb{R}$ and $A(\gamma_{1},\gamma_{2}, a),B(b)>0$ such that \[ K_{\Omega'} \bigl(\gamma_{1}(t),\gamma_{2}(t+T_{0})\bigr) \leq A e^{-Bt} \] for all sufficiently large $t$. Then, for any two geodesic rays \[ \sigma_{1},\sigma_{2}:[0,\infty)\to\Omega \] satisfying $\sigma_{1}(\infty)=\sigma_{2}(\infty),$ there exist $T\in\mathbb{R}$ and $B'>0$ such that \[ \limsup_{t\to\infty} \frac{\log K_{\Omega} \bigl(\sigma_{1}(t),\sigma_{2}(t+T)\bigr)}{t} \leq -B'. \]
 
\end{theorem}
In Example~\ref{Examp:E1}, we have illustrated how the family of quasi-geodesics considered in Theorem~\ref{P:exponetial} arises naturally. We now give an another application of Theorem~\ref{P:exponetial} to detect  the distance between a geodesic and its image under  holomorphic map. We are able to prove that 
\begin{theorem}\label{T:c(f)}
Let $f \in \mco(\bn , \bn)$ be a non-elliptic map (i.e., no fixed point in $\bn$) and $\xi \in \partial \bn$ be such  that the sequence of iterates of $f$ converges to $\xi$.   Let  $\gamma:[0,\infty) \to \bn$ be a geodesic with $\gamma(\infty)=\xi$.   Assume that there exists $\lambda \in (0, \infty)$ such that
\begin{itemize}
    \item[(i.)]
\[
\ln \lambda
=
\liminf_{z\to \xi}
\big[
K_{\bn}(\gamma(0),z)-K_{\bn}(\gamma(0),f(z))
\big],
\]
and $f(x_n)\to \xi$ as $n\to\infty$ whenever $x_n\to\xi$ is a sequence contained in the region
\[
A(\xi,R)
:=
\big\{
x\in\bn :
K_{\bn}(x,\eta)<R
\big\},
\]
where $\eta:[0,\infty)\to\bn$ is any geodesic ray with
$\eta(\infty)=\xi$, $R>0$ is arbitrary, and
\[
K_{\bn}(x,\eta)
:=
\inf_{s\geq 0}K_{\bn}(x,\eta(s)).
\]
    \item [(ii.)]
    There exists   $A,b > 0$  such that 
\begin{align*}\label{E:eqassum1}
 \big|K_{\mathbb{B}^d}(\gamma(0), f(\gamma(t))) - t - \ln{\lambda}\big| \le Ae^{-bt}\quad \forall t\geq 0.
\end{align*}
\end{itemize}
Then for  any geodesic ray $\sigma:[0, \infty) \to \bn$ with $\sigma(\infty)=\xi$ the following holds: 
\begin{align*}
\lim_{t\to \infty}\kb(f(\sigma(t)),\sigma(t-\ln{\lambda)})=0,     \end{align*}
\end{theorem}
\noindent
It turns out that, under the assumptions of Theorem~\ref{T:c(f)}, for every 
geodesic ray $\gamma$ in $\bn$ satisfying $\gamma(\infty)=\xi$, there exists a 
constant $m_{\gamma}>0$ such that the curve $f\circ\gamma|_{[t,\infty)}$ is a $(1,e^{-t})$-quasi-geodesic for every $t\geq m_{\gamma}$. We then apply 
Theorem~\ref{P:exponetial} to this family of quasi-geodesics and obtain a 
required estimates. Finally, applying \cite[Proposition~4.10]{AF2024} we get the desired conclusion of Theorem~\ref{T:c(f)}. To further illustrate the applicability of the theorem, we present an explicit example satisfying the hypotheses of 
Theorem~\ref{T:c(f)}(see Example~\ref{Examp:E1} ). This example also demonstrates how the family of 
 quasi-geodesics considered in Theorem~\ref{P:exponetial} 
arises naturally in this context. 

\smallskip

\noindent
We conclude the introduction with a brief description of the organization of the paper. In Section~\ref{S:prelims}, we collect the relevant notions and technical results that will be used throughout the paper. The proofs of Theorem~\ref{T:expcon} and Theorem~\ref{P:exponetial} are presented in  Section~\ref{S:expoproof}, and Sections~\ref{S:CATptoof} respectively. Section~\ref{S:appli} is devoted to the proofs of Theorem~\ref{T:Application} and Theorem~\ref{T:c(f)}. In Section~\ref{S:expo local to global}, we have discussed the exponential convergence property from certain  local-to-global and global-to-local perspectives. Finally, Section~\ref{S:appendix} contains the proof of an auxiliary result used in the proof of Theorem~\ref{P:exponetial}.

\section{Preliminaries and Technical results}\label{S:prelims}
In this section, we discuss the notations, definitions, and preliminary results that will be used throughout the paper. In the first subsection, we recall several basic notions needed in the sequel. In the second subsection, we present a number of known results that play a crucial role in the proof of our main theorem and also include proofs of  few auxiliary lemmas that will be needed later. The following notation will be used throughout the paper.

\medskip

\noindent
\textbf{Notations:}  For a domain $D \st \cn$ and $v \in \cn\setminus\{0\}$ 
\begin{align*}
   \partial_{D}(z)&:=\inf\{\|z-w\|: w \in \partial D\}. \\
   \partial_{D}(z;v)&:=
\inf\{\|z-w\| : w \in \partial D \cap (z+\mathbb{C}\cdot v)\}.
\end{align*}
For a metric space $(X,d)$ and $x,y \in X$, the closed geodesic segment joining $x, y$ will be denoted by $[x,y]$.  If 
$A\st X$ and $p\in X$  then we denote $d(p;A):=\inf\{d(p,a):a\in A\}$.
Let $\D $ denote the unit disc in $\cplx$ and $\rho_{\D}$ is the Poincar\'e distance.

\subsection{Preliminaries}
Recall that, for a domain $D \st \cn$, a point $z\in D$, and a vector $X\in\mathbb{C}^d$, the Kobayashi--Royden (pseudo)metric $\kappa_D$ is defined by
\[
\kappa_D(z; X) = \inf \{ |\alpha| : \exists \varphi \in \mathcal{O}(\D, D), \, \varphi(0) = z, \, \alpha \varphi'(0) = X \}.
\]
It follows from the \cite[Proposition~3]{Royden} that for any  domain $D \st \cn$ the map $\kappa_{D}: D \times \cn \to [0,\infty)$ is an upper semi-continuous.

\medskip

\noindent
The Kobayashi-Royden (pseudo) metric defines Kobayashi(pseudo) distance, which is defined as follow:

\medskip

\noindent
For $x, y \in D$, the Kobayashi (pseudo)distance  is defined as follows :
\begin{equation}
    K_D(x, y) = \inf_\gamma \int_0^1 \kappa_D(\gamma(t); \gamma'(t)) \, dt,
\end{equation}
where the infimum is taken over all piecewise $\smoo^1$ curves $\gamma : [0, 1] \to D$ with $\gamma(0) = x$ and $\gamma(1) = y$.
\begin{definition}\label{Def: Kobayashi hyperbolic}
A domain $D \st \cn$ is called Kobayashi hyperbolic domain if $(D, K_{D})$ forms a metric space.    
\end{definition}

Let us recall some properties of the Kobayashi–Royden metric on certain regular domains in $\cn$ that will be needed later. For more details on the Kobayashi distance, we refer the reader to \cite{Kobaya}.
\begin{itemize}
    \item[(A)] 
    For the unit ball in $\bn \st \cn$ the Kobayashi-Royden metric $\kappa_{\bn}(z,X)$ is a Riemann metric with  the real sectional curvature bounded above by $-1$.
    
    \medskip
    
    \item[(B)]
 For a bounded, strongly convex domain $D \st \cn$ with $\smoo^{\infty}$ smooth boundary the Kobayashi-Royden metric $\kappa_{D}$ is a $\smoo^{\infty}$  Finsler metric (see  \cite{Lempert81}).

\end{itemize}

\medskip

\noindent
Let $(D, K_{D})$ be a Kobayashi hyperbolic domain.  A \textit{geodesic} for $K_D$ is a curve $\sigma : I \to D$, where $I$ is an interval in $\mathbb{R}$, such that for any $s, t \in I$, 
\[
K_D(\sigma(s), \sigma(t)) = |s - t|.
\]
\noindent

Here we mention some known terminology:
\begin{itemize}
    \item If $x, y \in D$, $I = [0, L]$ and $\sigma(0) = x$, $\sigma(L) = y$, then $L = K_D(x, y)$ and we say that $\sigma$ is a geodesic joining $x$ and $y$.

\smallskip

 \item If $I = (-\infty, +\infty)$, we say that $\sigma$ is a \textit{geodesic line}.

 \smallskip
 
\item If $I = [0, \infty)$, we say that $\sigma$ is a \textit{geodesic ray}.

\smallskip

\item A geodesic ray in $D$ is said to be land at $\xi \in \partial D$ and denote it by $\gamma(\infty)=\xi$, if $\lim_{t\to \infty}\gamma(t)=\xi.$
\smallskip
    \item 
    A complex geodesic in $(D, K_{D})$ is a holomorphic map $\phi:\D \to D$ such that 
    \begin{align*}
        \rho_{\D}(s,t)&=K_{D}(\phi(s), \phi(t)).
    \end{align*}
\end{itemize}
\medskip
For $\lambda \geq 1$ and $\kappa \geq 0$,  a  $(\lambda,\kappa)$ {\em quasi-geodesic} in $D$ is a map $\sigma : I \to D$ such that 
$$ 
\frac{1}{\lambda}|t-s|-\kappa
\leq
K_{D}(\sigma(s),\sigma(t))
\leq
\lambda |t-s|+\kappa  \quad \forall s, t \in I.$$

\noindent
We now mention a result concerning the special properties of real and complex Kobayashi geodesics in strongly convex domains. These properties play a crucial role in the proof of Theorem~\ref{T:expcon}. For detailed accounts and proofs, we refer the reader to \cite[Lemma ~3.3, Lemma~2.6]{GauSesh2013}, \cite{Lemp1}, \cite{Lemp2}, \cite[Proposition~4.
9]{MR4978415}.

\begin{result}\label{R:Pgeodesic}
 Let $\Om \st \cn$ be a bounded strongly convex with $\smoo^{3}$ boundary. Then 
 \begin{enumerate}
     \item [(i.)]
     For any interval $I \st \re$ and a (real) geodesic $\alpha:I\to \Om$, there exists a complex geodesic $\phi:\D \to \Om$ such that $\alpha(I) \subset \phi(\D)$.

     \smallskip
     \item[(ii.)]
Every complex geodesic $\phi:\D \to\Omega$ extends as a $\mathcal{C}^{1,\alpha}$-smooth map from the closed unit disc $\overline{\D}$ for any $\alpha \in (0,1)$.

\smallskip

\item[(iii.)]
Suppose that $\phi_{1}, \phi_{2}:\D \to\Omega$ be two complex geodesics such that $\phi_{1}(1)=\phi_{2}(1)$ with $\phi_{1}'(1)=\lambda\phi_{2}'(1)$  for some $\lambda>0$ then $\phi_{1}(\D)=\phi_{2}(\D)$. 

\smallskip

\item [(iv.)]
If $\phi: \D \to \Om$ is a complex geodesic with $\phi(1)=\xi$ then $\langle\phi'(1),\nu_{\xi}\rangle>0$  where $\nu_{\xi}$ denotes the outer drawn normal vector on $\Om$ at the the point $\xi.$

 \end{enumerate}
\end{result}

\noindent

\noindent
If $D \st \cn$ is a Kobayashi hyperbolic domain and $z,w,o \in D$ then the {\em Gromov Product} of $z,w$ with the base point $o$ is denoted by $\langle z\mid w\rangle_{o}^{D}$ and defined as follows 
\[
\langle z \mid w\rangle_{o}^{D}
:= \frac{1}{2}\bigl(K_{D}(z,o)+K_{D}(w,o)-K_{D}(z,w)\bigr),
\qquad z,w\in D.
\]

\medskip

\noindent
Next, we discuss the notion of a ${\rm CAT}(-1)$ space, which provides a synthetic formulation of  negative curvature for general metric spaces.  Before recalling the definition  of ${\rm CAT}(-1)$ space let us we mention the notion of {\em comparison triangle} and {\em comparison point}.
Given a geodesic triangle
\[
\triangle=[x,y]\cup[y,z]\cup[z,x]
\]
in $(X,d)$, a \emph{comparison triangle} for $\triangle $ in $(\D, \rho_{\D})$ is defined by geodesic triangle $\overline{\triangle}$
\[
\overline{\Delta}=[\bar{x},\bar{y}]\cup[\bar{y},\bar{z}]
\cup[\bar{z},\bar{x}],
\]
is  in $(\D, \rho_{\D})$ whose side lengths agree with those
of $\triangle$, namely,
\[
d(x,y)=\rho_{\D}(\bar{x},\bar{y}), \qquad
d(y,z)=\rho_{\D}(\bar{y},\bar{z}), \qquad
d(z,x)=\rho_{\D}(\bar{z},\bar{x}).
\]

\medskip

\noindent
A
\emph{comparison point} of  a point $p\in [x,y]$ is  a point $\bar p\in [\bar x,\bar y]$ such that
\[
d(x,p)=\rho_{\D}(\bar x,\bar p).
\]
Comparison points on the other sides are defined analogously.

\medskip

\begin{definition}(\textbf{CAT($\mathbf{-1}$) spaces})\label{Def:cat}
A metric space $(X,d)$ is called a
CAT$(-1)$ space if it satisfies the following conditions.
\begin{itemize}
    \item 
    $X$ is a  geodesic metric space (i.e. for all $x,y \in X$ an isometric embedding of an interval of length $d(x,y)$ in $\mathbb{R}$ into $X$). 
    \smallskip
    \item 
  For every geodesic triangle $\triangle (p,q,r)$ there exists a comparison triangle $\bar{\triangle}(\bar{p},\bar{q},\bar{r})$  in the space $(\D, \rho_{\D})$  such that if  $\bar{x}, \bar{y} \in \bar{\triangle}(\bar{p},\bar{q},\bar{r})$ are   comparison points of  $x,y \in \triangle (p,q,r)$ respectively, then
  \begin{align*}
    d(x,y)\leq \rho_{\D}(\bar{x}, \bar{y}).  \end{align*}
\end{itemize} 
\end{definition}
\noindent
It follows directly from the definition that $(\D, \rho_{\D})$ is a prototypical example of a ${\rm CAT}(-1)$ space. From a remarkable result due to Alexandrov and property $(A)$, it follows
\begin{result}\cite{Bridsonbook}\label{R:alexnder}
A complete, simply
connected Riemannian manifold is a {\rm CAT}$(-1)$ space if and only if its
sectional curvature is bounded above by $-1$. In particular, $(\bn, K_{\bn})$ is a ${\rm CAT}(-1)$ space.
\end{result}

\begin{proposition}\label{Rem:CAT}
Let $\Om \St \cn$ be a  strongly convex domain with $\smoo^{\infty}$ boundary. Then $(\Om, K_{\Om})$ forms a ${\rm CAT}(-1)$ space if and only if $\Om \cong \bn$.
\end{proposition}
\begin{proof}
If $\Om \st \cn$ is a bounded strongly convex domain with $\smoo^{\infty}$ boundary then $\kappa_{D}$ is a $\smoo^{\infty}$- Finsler metric (see Property (B)). Suppose that $(\Om, K_{\Om})$ forms a  {\rm CAT(-1)} space. Then for any $z \in \Om$, $T_{z}\Om$ is the tangent cone  at $z$. It follows from \cite[Theorem~1.9]{Bridsonbook} that $T_{z}\Om$ with norm $\kappa_{\Om}(z,\cdot):T_{z}\Om \to \mathbb{R}$ forms a ${\rm CAT(0)}$ space.  Then from the ${\rm CAT(0)}$ inequality it turns out that  the norm $\kappa_{\Om}(z, \cdot)$ comes from a inner product. Consequently, $\kappa_{\Om}(z, \cdot)$ is Riemannian for all $z \in \Om$. Hence, for every $z\in\Omega$, there exists a real inner product $g_z$ on $T_z\Omega$ such that
$$
\kappa_{\Omega}(z;v)^2=g_z(v,v)
$$
for every $v\in T_z\Omega$. Since the Kobayashi--Royden metric $\kappa_{\Omega}$ is smooth on $T\Omega\setminus{0}$, the family of inner products ${g_z}(,)$ varies smoothly with $z$, and hence defines a smooth Riemannian metric $g$ on $\Omega$. Moreover, by the complex homogeneity of the Kobayashi--Royden metric,
$$
\kappa_{\Omega}(z;iv)=\kappa_{\Omega}(z;v).
$$
Consequently,
$
g_z(iv,iv)=g_z(v,v).
$
By polarization, we obtain
$
g_z(iv,iw)=g_z(v,w)
$
for every $v,w\in T_z\Omega$. Thus $g$ is $J$-invariant, and consequently the Kobayashi metric is a smooth Hermitian metric. 
Then the inverse of the Lempert representation map (see \cite{Lempert81}) defines a $\mathcal{C}^{\infty}$-diffeomorphism from the unit ball $\mathbb{B}^{d}$ onto the domain $\Omega$, which is holomorphic along every complex disc passing through the origin. Hence, by Forelli's theorem (see \cite{Patrizio1983}), it is a biholomorphism (see also  \cite[Chapter~7, page 193]{GreeneKimKrantz2011} for detail).

\end{proof}

\noindent
We now discuss about the notions which is required  for Theorem~\ref{T:Application}. 

\medskip

\noindent
{\bf Squeezing function:} One natural intrinsic measure of the complex geometry of a domain is the squeezing function. Let $D\subset \mathbb{C}^d$ be a bounded domain. The {\em squeezing function} of $D$ is denoted by
$s_D:D\to (0,1]$ and is defined by
\[
s_D(z)
=
\sup\Big\{r>0:\ \exists  \ \text{an injective map }f \in \mco(D, \bn)~~
 {\rm with }~f(z)=0,~~ r\mathbb{B}^d\subset f(D)\Big\}.
\]
\noindent
The notions of the squeezing function evolved from a closely related notion introduced in the works of Liu et al. \cite{LiuYauTung2004} \cite{LiuYauTung2005}. Its first formal definition was subsequently given by Deng, Guan, and Zhang \cite{DenGuZha2012}.

\medskip

\noindent
Recall that a domain $D \subset \mathbb{C}^d$ is called \emph{balanced} if, for every $z \in D$ and every $\lambda \in \mathbb{C}$ with $|\lambda| \leq 1$, it follows that $\lambda z \in D.$ The \emph{Minkowski functional} of $D$ is defined by $h_D(z):=\inf \left\{ t>0 : t^{-1}z \in D \right\}.$ Consequently, we get $D=\left\{ z\in\mathbb{C}^d : h_D(z)<1 \right\}$. When $D$ is a balanced convex domain, the Minkowski functional $h_D$ defines a norm on $\mathbb{C}^d$ (see \cite[Lemma 3.2]{Rudin2008}). For $r>0$, we set $D(r):=\left\{ z\in\mathbb{C}^n : h_D(z)<r \right\}$. Let $D$ be a balanced convex domain in $\mathbb{C}^{d}$, and let $\Omega\subset\mathbb{C}^d$ be a bounded domain. The \emph{ generalized squeezing function} of $\Omega$ with respect to $D$, denoted by $T_{\Omega}^{D}$, is defined by
\[
T_{\Om}^{D}(z)
:=
\sup \Bigl\{
r>0 :
D(r)\subset f(\Omega),\;
f\in \mathcal{O}(\Omega,D),\;
f(z)=0,\;
f \text{ is injective}
\Bigr\}.
\]
Observe that when $D=\mathbb{B}^d$, the function $T_{\Om}^{D}$ coincides with the classical squeezing function $s_{\Om}$. The reader is referred to \cite{GuptaPant2023,CM2024} for a detailed account of the properties and applications of the generalized squeezing function.

\medskip

\noindent

As noted earlier, $\mathbb{X}_d$ denotes the collection of all convex domains in $\mathbb{C}^d$ that do not contain any complex affine line, and $\mathbb{X}_{d,0}$ denotes the set of pairs $(\Omega,x)$ with $\Omega \in \mathbb{X}_d$ and $x \in \Omega$. It follows from a theorem of Barth (see \cite[Theorem~1]{Barth1980}) that every domain in $\mathbb{X}_d$ forms a complete metric space  with respect to the Kobayashi distance.

We now describe a natural topology on the sets $\mathbb{X}_d$ and $\mathbb{X}_{d,0}$. Given two compact sets $A,B \subset \mathbb{C}^d$, define the Hausdorff distance between them to be
\[
d_H(A,B)
=
\max \left\{
\max_{a\in A}\min_{b\in B}\|a-b\|,
\;
\max_{b\in B}\min_{a\in A}\|b-a\|
\right\}.
\]
\noindent
The Hausdorff distance is a complete metric on the set of compact subsets of $\mathbb{C}^d$. To consider general closed sets, we introduce the local Hausdorff semi-norms between two closed sets $A,B \subset \mathbb{C}^d$ by defining
\[
d_H^{(R)}(A,B)
=
d_H\!\left(
A\cap \overline{B_d(0;R)},
\,
B\cap \overline{B_d(0;R)}
\right)
\]
for $R>0$. Since an open convex set is determined by its closure, we can define a topology on $\mathbb{X}_d$ and $\mathbb{X}_{d,0}$ using these seminorms as follows:

\begin{itemize}

\item A sequence $C_n \in \mathbb{X}_d$ converges to $C \in \mathbb{X}_d$ if there exists some $R_0 \geq 0$ such that
\[
d_H^{(R)}(C_n,C)\to 0
\]
for all $R\geq R_0$.

\item A sequence $(C_n,x_n)\in \mathbb{X}_{d,0}$ converges to $(C,x)\in \mathbb{X}_{d,0}$ if $C_n$ converges to $C$ in $\mathbb{X}_d$ and $x_n$ converges to $x$ in $\mathbb{C}^d$.
\end{itemize}

The following lemma can be established by adapting the arguments used in \cite[Proposition~7.1]{Zimmer18Gap} and \cite[Lemma~3.3]{FathiHouci2023}. Since the proof is essentially identical to those references, we omit the details.

\begin{lemma}\label{L:gneralsqu}
   Let $(\Om_{n},x_{n}) \in \mathbb{X}_{d,0}$ such that $(\Omega_{n},x_{n}) \to (\Om_{\infty}, x_{\infty})$ in the local Hausdorff topology of $\mathbb{X}_{d,0}$ and $(D,0)\in \mathbb{X}_{d,0}$ be a fixed balanced bounded convex domain. Then $$\limsup_{n \to \infty} T^{D}_{\Om_{n}}(x_{n}) \leq  T^{D}_{\Om_{\infty}}(x_{\infty}).$$ 
\end{lemma}

\medskip

\subsection{Technical Results}
We begin by recalling the Morse lemma, which will be used repeatedly in the sequel.
 \begin{result}\cite[Theorem~1.7]{Bridsonbook}\label{R:Morselemma}
Suppose that $(X,d)$ is a $\delta$-hyperbolic space. For any $A \ge 1$ and $B \ge 0$
there exists $R:=R(\delta,A,B)>0$ such that if
$
\gamma:[a_1,b_1]\to X
$
is a geodesic and
$
\sigma:[a_2,b_2]\to X
$
is an $(A,B)$-quasi-geodesic with
$
\gamma(a_1)=\sigma(a_2)
\quad \text{and} \quad
\gamma(b_1)=\sigma(b_2),
$
then
\[
\max \left\{
\max_{t\in [a_1,b_1]}
d\bigl(\gamma(t),\sigma([a_2,b_2])\bigr),
\,
\max_{t\in [a_2,b_2]}
d\bigl(\sigma(t),\gamma([a_1,b_1])\bigr)
\right\}
\le R.
\]
\end{result}
The next result is Royden's localization result for Kobayashi-metric metric. We need the result in the proof of Theorem~\ref{T:local to global} and Theorem~\ref{T:Glob to loc}.

\begin{result}\label{R:Roydenestimate}\cite[Lemma~2]{Royden1971}
Let $D \st \cn$ be a Kobayashi hyperbolic domain and $U \subset \cn$ be a open set such that $U \cap D \neq \emptyset$ and $U \cap D$ is connected. Then 
\begin{align}\label{E:Roydenestimate}
    \kappa_{D \cap U}(z,v) \leq \coth{\big(K_{D}(z; D\setminus U)\big)}\kappa_{D}(z,v)\,\,\forall z \in D \cap U~\,\,\,\forall v\in \cn,
\end{align}
where $K_{D}(z;D\setminus U):=\inf\{K_{D}(z,w):w\in D\setminus U\}$.
\end{result}

\smallskip

\noindent
We sate the following result from \cite[Lemma~2.1]{BalogBonk2000} which we need to prove Theorem~\ref{T:expcon}.
\begin{result}\label{R:BalogBonk}
Suppose that $\Omega$ is a bounded domain in $\mathbb{C}^{d}$ with $\smoo^{2}$-smooth
boundary. Let $(a,b)$ denote the line segment joining $a,b\in\mathbb{C}^{d}$ and,
for $\eps>0$, let  $N_{\eps}
:=
\bigcup_{p\in\partial\Omega}
\left(p-\eps\nu_{p},\,p+\eps\nu_{p}\right),$
where $\nu_{p}$ is the outer unit normal to $\partial\Omega$ at the point $p$.
Then there exists a sufficiently small $\eps>0$ such that:
\begin{enumerate}
    \item[(i.)] For all $z\in N_{\eps}$ there exists a unique point
    $\pi_{\Omega}(z)\in\partial\Omega$ such that
    \[
    \|z-\pi_{\Omega}(z)\|=\delta_{\Omega}(z).
    \]

    \item[(ii.)] The signed boundary distance function
    \[
    \rho_{\Omega}:\mathbb{C}^{d}\to\mathbb{R}
    \]
 defined by $\rho_{\Om}(z)=-\partial_{\Om}(z)$ for $z \in \Om$ and $\rho_{\Om}(z)=\partial_{\Om}(z)$ if $z \in \Om^{c}$    is $\smoo^{2}$-smooth on $N_{\eps}$.
 \item[(iii.)] The projection map
    \[
    \pi_{\Omega}:N_{\eps} \to\partial\Omega
    \]
    is $\smoo^{1}$-smooth on $N_{\eps}$.
    \end{enumerate}
\end{result}

The following result due to Kosi\'nski--Nikolov--\"Okten  provides a sharp estimate of the Kobayashi distance in bounded strongly convex domain with $\smoo^{2, \alpha}$ a boundary for $\alpha \in (0,1]$. The result will be used in the proof of Theorem~\ref{T:expcon}. 
\begin{result}[Paraphasing  ({\cite[Theorem~1]{nikothmokt}})]\label{R:nikolov-thomas}
  Let $D\st \cn$ be a strongly pseudoconvex domain with $\smoo^{2,\alpha}$ boundary for some $\alpha \in (0, 1]$ and  $p\in \partial D$.  For $z,w \in D$ sufficiently close to $p$, define
\begin{equation}\label{eq:BD}
B_D(z,w)
=
|\big\langle z-w,\, \nu_{\pi(z)} \big\rangle|
+
\|z-w\|^{2}
+
\|z-w\|
\sqrt{\delta_D(z)},
\end{equation}
where $\nu_{\pi(z)}$ denotes the outward unit normal to $\partial D$ at the boundary point $\pi(z)$. Assume that
\begin{equation}\label{eq:AD}
A_D(z,w)
=
\frac{B_D(z,w)}
{\sqrt{\delta_D(z)\delta_D(w)}}.
\end{equation}
\noindent
There exist constants $0<c<C$(depending on $D$)  such that
\begin{equation}\label{eq:NT-estimate}
\log\!\left(1+cA_D(z,w)\right)
\leq
K_D(z,w)
\leq
\log\!\left(1+C A_D(z,w)\right),
\qquad
z,w\in D.
\end{equation}
\end{result}

\section{Proof of Theorem~\ref{P:exponetial}}\label{S:CATptoof}
In this section, we is present the proof of Theorem~\ref{P:exponetial}. 
\begin{proof}[Proof of Theorem~\ref{P:exponetial}]

We assume that $\xi=e_{1}$.
It is given that, for $j\in\{1,2\}$, the maps 
$\gamma_j:[0,\infty)\to\mathbb{B}^d
$ are rays satisfying
$
\gamma_1(\infty)=\gamma_2(\infty)=e_{1}$,
and 
\begin{align}\label{E:ae1}
|s-t|-ae^{-bs}
\leq
K_{\mathbb{B}^d}\bigl(\gamma_j(s),\gamma_j(t)\bigr)
\leq
|s-t|+ae^{-bs}\quad \forall t>s \geq 0.
\end{align}
\noindent
For sufficiently large $n\in\mathbb{N}$, we consider the three points on the image of the quasi-geodesic $\gamma_1$ defined by
\[
P:=\gamma_1(s), \qquad
Q:=\gamma_1(2s), \qquad
R:=\gamma_1(2s+n).
\]
Let $\triangle PQR$ denote the geodesic triangle in the metric space $(\mathbb{B}^d,K_{\mathbb{B}^d})$. Furthermore, let
$$
\sigma_s^n:[0,K_{\mathbb{B}^d}(P,R)]\to\mathbb{B}^d,\qquad
\eta_{1,s}:[0,K_{\mathbb{B}^d}(P,Q)]\to\mathbb{B}^d,
\qquad
\eta_{2,s}:[0,K_{\mathbb{B}^d}(Q,R)]\to\mathbb{B}^d
$$
be the arc-length parametrized geodesics representing the sides $PR$, $PQ$, and $QR$, respectively, of the geodesic triangle $\triangle
PQR$.

\medskip

\noindent
{\bf Step-I. Showing that { $\mathbf{\kb\big(\gamma_{1}(2s), \sigma^{n}_{s}(s)\big)\leq A e^{-Bs}}$ for some $\mathbf{A,B>0}$.}}
We infer from Result~ \ref{R:alexnder}   that  $(\bn, \kb)$ is  a ${\rm CAT}(-1)$ space. Therefore, for the geodesic triangle $\triangle PQR$,   there exists a comparison triangle $\triangle \bar{P}\bar{Q}\bar{R}$ in $(\mathbb{D}, \rho_{\mathbb{D}})$ (see  figure ~\ref{fig:1}). Here $\rho_{\D}(0,z)=2\tanh^{-1}(|z|)$.
\begin{center}

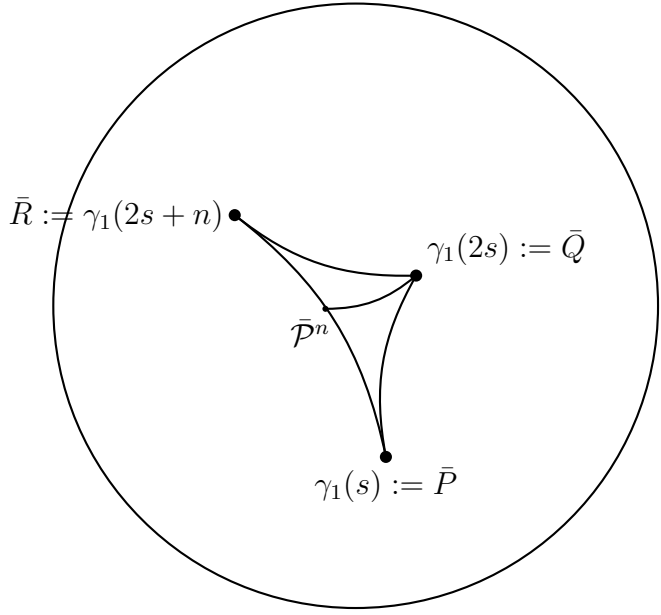
\begin{figure}[H]
\centering
\begin{tikzpicture}[scale=4]

\draw[thick] (0,0) circle (1);

\coordinate (A) at (0.2, 0.1);
\coordinate (B) at (-0.4, 0.3);
\coordinate (C) at (0.1, -0.5);
\coordinate (D) at (-0.1, -0.01);
\fill (A) circle (0.02);
\fill (B) circle (0.02);
\fill (C) circle (0.02);
\fill (D) circle (0.01);
\node[above right] at (A) {$\gamma_{1}(2s):=\bar{Q}$};
\node[left] at (B) {$\bar{R}:=\gamma_{1}(2s+n)$};
\node[below] at (C) {$\gamma_{1}(s):=\bar{P}$};

\draw[thick, black] (A) to[bend left=20] (B);
\draw[thick, black] (B) to[bend left=20] (C);
\draw[thick, black] (C) to[bend left=20] (A);
\draw[thick, black] (A) to[bend left=20] (D);

\node at ($(C)!0.5!(B)+(0,0.01)$) {$\mathcal{\bar{P}}^n$};


\path let
    \p1 = (B),
    \p2 = (C),
    \p3 = (A)
in
coordinate (Foot) at ($(B)!{((\x3-\x1)*(\x2-\x1)+(\y3-\y1)*(\y2-\y1))/
((\x2-\x1)^2+(\y2-\y1)^2)}!(C)$);
\end{tikzpicture}
\caption{$\triangle \bar{P}\bar{Q}\bar{R}$}
\label{fig:1}
\end{figure}
\end{center}
\noindent
 It follows from the definition of the comparison triangle that $$\rho_{\D}(\bar{P}, \bar{R})=\kb(P,R),~~\rho_{\D}(\bar{P}, \bar{Q})=\kb(P, Q),\,\rho_{\D}(\bar{Q}, \bar{R})=\kb(Q, R).$$
Let $\bar{\sigma}^{n}_{s}$, $\bar{\eta}_{1,s}$, $\bar{\eta}_{2,s}$ denotes the corresponding geodesics in $\D$. Note that by construction we have $$\kb\big(P, \sigma^{n}_{s}(s)\big)=s=\rho_{\D}\big(\bar{P}, \bar{\sigma}^{n}_{s}(s)\big).$$ Hence,  $\bar{\sigma}^{n}_{s}(s)$ is comparison point for ${\sigma}^{n}_{s}(s)$. Consequently, it follows from the definition of ${\rm CAT}(-1)$ space  that  
\begin{align}\label{E:ae0}
    \kb(\gamma_{1}(2s), {\sigma}^{n}_{s}(s))\leq \rho_{\D}(\bar{\gamma_{1}}(2s), \bar{\sigma}^{n}_{s}(s))\,\, \forall s \geq 0.
\end{align}
\noindent
We show that $\rho_{\D}(\bar{\gamma_{1}}(2s), \bar{\sigma}^{n}_{s}(s))\leq Ae^{-Bs}~~ \forall s\geq 0$  for some $A, B>0$. Let $\mathcal{\bar{P}}^n:=\bar{\sigma}^{n}_{s}(a^{n}(s))$ be the foot of the orthogonal projection from the point $\bar{Q}$ on the side $[\bar{P}, \bar{R}]$. Invoking the hyperbolic laws of sine in the triangle $\triangle \bar{P}\bar{Q}\bar{\mathcal{P}^n}$, we obtain that 
\begin{align}\label{E:ae2}
\frac{\sinh{\rho_{\D}\big(\bar{P}, \bar{Q}})}{\sin{\angle \bar{Q}\bar{\mathcal{P}^n}}\bar{P}}=\frac{\sinh{\rho_{\D}\big(\bar{P}, \bar{\mathcal{P}^n}}\big)}{\sin{\angle \bar{P}\bar{Q}\bar{\mathcal{P}^n}}}=\frac{\sinh{\rho_{\D}\big(\bar{Q},\bar{\mathcal{P}^n}}\big)}{\sin{\angle \bar{\mathcal{P}^n}\bar{P}\bar{Q}}}.
\end{align}
\noindent
Taking $s>0$ and $n\in\mathbb N$ sufficiently large, the orthogonal projection $\bar{\mathcal P}^{n}$ of $\bar Q$ onto the geodesic $[\bar P,\bar R]$ lies in the interior of the segment $[\bar P,\bar R]$, as will also follow from the estimates below. Since $\bar{\mathcal{P}^n}$ is the orthogonal projection of the point $\bar{Q}$ on the side $[\bar{P}, \bar{R}]$, we have $\angle \bar{Q}\bar{\mathcal{P}^n}\bar{P}=\frac{\pi}{2}$. Therefore, from \eqref{E:ae2}, we deduce the following
\begin{align}\label{E:ae3}
\sinh\big(\rho_{\D}(\bar{P}, \bar{\mathcal{P}^n})\big)
&= \sin\big(\angle \bar{P}\bar{Q}\bar{\mathcal{P}^n}\big)\cdot \sinh\big(\rho_{\D}(\bar{P}, \bar{Q})\big)\notag\\
\sinh{\big(\rho_{\D}(\bar{P}, \bar{\mathcal{P}^n})\big)}
&\leq \sinh\big(\rho_{\D}(\bar{P}, \bar{Q})\big)\notag\\
\rho_{\D}\big(\bar{P}, \bar{\mathcal{P}^n}\big)
&\leq \rho_{\D}\big(\bar{P}, \bar{Q}\big)\notag\\
\rho_{\D}\big(\bar{\sigma}^{n}_{s}(0), \bar{\sigma}^{n}_{s}(a^{n}(s))\big)
&\leq \kb(P,Q)\notag\\
a^{n}(s) &\leq s + a e^{-bs} 
\end{align}
Similarly, using the hyperbolic sine law on the triangle $\triangle \bar{Q}\bar{\mathcal{P}^{n}}\bar{R}$ we obtain the following:
\begin{align}\label{E:ae4}
\sinh{\rho_{\D}\big(\bar{\mathcal{P}^n}, \bar{R}\big)}&\leq    \sinh{\rho_{\D}\big(\bar{Q}, \bar{R}\big)}\notag\\
\rho_{\D}\big(\bar{\mathcal{P}^n}, \bar{R}\big)&\leq\rho_{\D}\big(\bar{Q}, \bar{R}\big)\notag\\
\rho_{\D}\big(\bar{\mathcal{P}^n}, \bar{R}\big)&\leq K_{\bn}(Q,R)\notag\\
\rho_{\D}\big(\bar{\mathcal{P}^n}, \bar{R}\big)&\leq n+a e^{-2bs}\notag\\
\rho_{\D}(\bar{P}, \bar{R})-\rho_{\D}(\bar{\mathcal{P}^n}, \bar{R})&\geq \rho_{\D}(\bar{P}, \bar{R})-n-a e^{-2bs}\notag\\
\rho_{\D}(\bar{P},\bar{\mathcal{P}^n})&\geq s-2ae^{-bs}\notag\\
a^{n}(s)&\geq s-2a e^{-bs}.
\end{align}
Now combining \eqref{E:ae3}, \eqref{E:ae4} we get \begin{align}\label{E:ae5}
    |a^{n}(s)-s|\leq 2ae^{-bs}.
\end{align}
We now claim the following fact:

\smallskip

\noindent 
\begin{claim*}\label{claim:squaregromov}\footnote{This claim is a standard fact. For the sake of completeness, and since we are unaware of an explicit reference, we provide a proof of the claim Section~\ref{S:appendix} (see  Lemma~\ref{L:AppenLemma}).  }  There exists $A_{1},M_{1}>0$ such that \begin{align}\label{E:ae6}
    \rho_{\D}(\bar{Q}, \bar{\sigma}^{n}_{s}(a^{n}(s)))\leq A_{1}\sqrt{\big\langle\bar{P}\mid \bar{R}\big\rangle_{\bar{Q}}^{\D}}\,\qquad\forall s \geq M_{1}.
\end{align}  
\end{claim*}

We now find an upper bound of $\big\langle\bar{P}\mid \bar{R}\big\rangle_{\bar{Q}}^{\D}$. 
\begin{align}\label{E:ae7}
  \bigg\langle\bar{\gamma}_{1}(s)\mid \bar{\gamma}_{1}(2s+n)\bigg\rangle_{\bar{\gamma}_{1}(2s)}^{\D}  &\leq \frac{1}{2}\bigg[\rho_{\D}(\bar{\gamma}_{1}(s),\bar{\gamma}_{1}(2s))+\rho_{\D}(\bar{\gamma}_{1}(2s+n),\bar{\gamma}_{1}(2s))-\rho_{\D}(\bar{\gamma}_{1}(s),\bar{\gamma}_{1}(2s+n))\bigg]\notag\\
  &\leq \frac{1}{2}\bigg[s+ae^{-bs}+n+ae^{-2bs}-(s+n-ae^{-b s})\bigg]\notag\\
  &\leq \frac{a}{2} \bigg[e^{-bs}+e^{-2bs}+e^{-bs}\bigg]\notag\\
  &\leq 2a e^{-bs}.
\end{align}
Hence, it follows from \eqref{E:ae6} and \eqref{E:ae7} we get
\begin{align}\label{E:ae8}
    \rho_{\D}(\bar{\gamma}_{1}(2s), \bar{\sigma}^{n}_{s}(a^{n}(s)))&\leq a'.e^{-b's} \quad \forall s\geq M_{1},
\end{align}
where $a'=\sqrt{2a}A_{1}$ and $b'=b/2$.

\smallskip
\noindent
In view of  \eqref{E:ae8} and \eqref{E:ae5} we conclude that the following holds for all $s>M_{1}$ :
\begin{align}
\rho_{\D}(\bar{\gamma}_{1}(2s),\bar{\sigma}^{n}_{s}(s))&\leq  \rho_{\D}(\bar{\gamma}_{1}(2s), \bar{\sigma}^{n}_{s}(a^{n}(s)))+\rho_{\D}(\bar{\sigma}^{n}_{s}(a^{n}(s)),\bar{\sigma}^{n}_{s}(s))\notag\\
&\leq a'e^{-b's}+|a^{n}(s)-s|\notag\\
&\leq a'e^{-b's}+2ae^{-bs}\leq 2Ae^{-Bs},
\end{align}
where $A=\max\{a', 2a\}$ and $B=b/2$. This completes the proof of Step I. \hfill $\blacktriangleleft$

\medskip

\noindent
Note that the sequence of geodesic segments $\sigma^{n}_s : [0, \kb(P, R)] \to \mathbb{B}^{d}$  converges (upto a subsequence) locally uniformly on $[0, \infty)$ to a geodesic ray $\sigma_{s}:[0, \infty) \to \bn$ with $\sigma_{s}(0)=\gamma_{1}(s)$ and $\sigma_{s}(\infty)=\gamma_{1}(\infty)$ as $n\to \infty$. Let $\tau_{s}:[0, \infty) \to \bn$ be the geodesic ray such that $\tau_{s}(0)=\gamma_{2}(s)$ and $\tau_{s}(\infty)=\gamma_{2}(\infty)$. Since $\bn$ is ${\rm CAT}(-1)$ space, it follows from the exponential convergence property and the ${\rm  CAT}(-1)$ inequality that there exists $T_{s} \in \mathbb{R}$ such that  
\begin{align} \label{E:cat1}
K_{\bn}(\sigma_{s}(t),\tau_{s}(t+T_{s})) \leq e^{-t} \sinh{\bigg(\frac{K_{\bn}(\sigma_{s}(0), \tau_{s}(0))}{2}\bigg)}  =e^{-t}\sinh{\bigg(\frac{K_{\bn}(\gamma_{1}(s), \gamma_{2}(s))}{2}\bigg)},  
\end{align}
$ \forall t \geq \max\{0, -T_{s}\}$ (see \cite[Proposition ~4.1]{HeintzeImHof1977}). As will be shown in Step-III  that $T_s\to T_0$ as $s\to\infty$ for some $T_0\in\mathbb{R}$. Hence, for all sufficiently large $s$, $|T_s-T_0|<1/2$, and therefore $\max\{0,-T_s\}\geq \max\{0,-T_0+1/2\}$.

\smallskip

\noindent
We now show  there exists $A_{2}>0$ such that $K_{\bn}(\gamma_{1}(s), \gamma_{2}(s))<A_{2}$. Let $\tilde{\gamma}_{j}:[0, \infty) \to \bn$ be two geodesic ray with $\tilde{\gamma}_{j}(0)=\gamma_{j}(0)$ and $\tilde{\gamma}_{j}(\infty)=\gamma_{j}(\infty)$ for all $j \in \{1,2\}$. Then it follows from the  Morse lemma (Result~\ref{R:Morselemma}) there exists constant $C>0$ (independent of the endpoints) such   that for all $t\geq 0$ there exists $a_{j}(t)\geq 0$ such that the following holds 
\begin{align}\label{E:morseapp}
\kb(\tilde{\gamma}_{j}(a_{j}(t)),\gamma_{j}(t)) <C\quad \forall t \geq 0.
\end{align}
We now observe the following:

\begin{align*}
\kb(\tilde{\gamma}_{j}(0), \gamma_{j}(t))-\kb(\tilde{\gamma}_{j}(a_{j}(t)), \gamma_{j}(t))&\leq \kb(\tilde{\gamma}_{j}(a_{j}(t)), \tilde{\gamma}_{j}(0))\notag\\
t-(C+1)&\leq a_{j}(t).
\end{align*}
 Also we have 
\begin{align*}
a_{j}(t)=\kb(\tilde{\gamma}_{j}(0), \tilde{\gamma}_{j}(a_{j}(t))&\leq\kb(\tilde{\gamma}_{j}(0), \gamma_{j}(t))+\kb(\gamma_{j}(t), \tilde{\gamma}_{j}(a_{j}(t))) \leq t+(C+1).
\end{align*}
Therefore, from the above observation we have
\begin{align}\label{E:new0}
    |a_{j}(t)-t|\leq 1+C \quad \forall t \geq 0\quad \forall j \in \{1,2\}.
\end{align}
We now deduce that 
\begin{align}
    \kb(\gamma_{1}(s), \gamma_{2}(s))&\leq \kb(\gamma_{1}(s), \tilde{\gamma}_{1}(a_{1}(s)))+\kb(\tilde{\gamma}_{2}(a_{2}(s)), \tilde{\gamma}_{1}(a_{1}(s)))+\kb(\tilde{\gamma}_{2}(a_{2}(s)), \gamma_{2}(s))\notag\\
    &\leq 2C+\kb(\tilde{\gamma}_{2}(a_{2}(s)), \tilde{\gamma}_{1}(a_{1}(s)))\notag\\
   &\leq  2C+\kb(\tilde{\gamma}_{2}(a_{2}(s)), \tilde{\gamma}_{1}(a_{2}(s)))+\kb(\tilde{\gamma}_{1}(a_{1}(s)), \tilde{\gamma}_{1}(a_{2}(s)))\notag\\
   &=2C+C_{1}+|a_{2}(s)-a_{1}(s)|\notag\\
   &=2C+C_{1}+|a_{2}(s)-s|+|s-a_{1}(s)|<4C+2+C_{1},
\end{align}
where $0<C_{1}=\sup_{t\geq0}\kb(\tilde{\gamma}_{1}(t), \tilde{\gamma}_{2}(t))<\infty$. The constant $C_{1}>0$ is finite as $\tilde{\gamma}_{j}$ represents same Gromov boundary.   

\noindent
Therefore,  taking $t=s$ in  \eqref{E:cat1} we obtain  
\begin{align}\label{E:ae9}
 K_{\bn}(\sigma_{s}(s),\tau_{s}(s+T_{s}))& \leq A_{2}e^{-s}\quad \forall s\geq 0,  
\end{align}
where $A_{2}=\sinh(\frac{3C+1+C_{1}}{2})$.

\medskip

\noindent
Next, we apply a similar method as Step-I to the quasi-geodesic $\gamma_{2}$. For suitable large $n \in \mathbb{N}$, we  consider three points  on the image of quasi-geodesic  $\gamma_{2}$  defined as follows:  $$\mathcal{A}:=\gamma_2(s), ~\mathcal{B}:=\gamma_{2}(2s+T_{s}), ~\mathcal{C}:=\gamma_{2}(2s+n).$$
We will prove in Step-III that the  $T_{s} \to T_{0} \in \mathbb{R}$ as $s \to \infty$. Therefore, we can choose $M_{2} >0$ such that $2s+T_{s}>s$ for all $s \geq M_{2}$. We now consider the geodesic triangle $\triangle \mathcal{A}\mathcal{B}\mathcal{C}$ in the metric space $(\bn, K_{\bn})$. We assume that 
$$\tau_s^n : [0, \kb(\mathcal{A}, \mathcal{C})] \to \mathbb{B}^{d},\,\, \tau_{1,s}:[0, \kb(\mathcal{A}, \mathcal{B})] \to \mathbb{B}^{d},\,\,\tau_{2,s}:[0, \kb(\mathcal{B}, \mathcal{C})] \to \mathbb{B}^{d} $$
\noindent
denotes the arc-length parametrized geodesic in $\bn$ represents the sides $\mathcal{A}\mathcal{C}$, $\mathcal{A}\mathcal{B}$ and $\mathcal{B}\mathcal{C}$  respectively in the geodesic triangle $\triangle \mathcal{A}\mathcal{B}\mathcal{C}$.
\\

\medskip
\noindent
{\bf Step-II. Showing that  { $\mathbf{\kb\big(\gamma_{2}(2s+T_{s}), \tau^{n}_{s}(s+T_{s})\big)\leq A_{3} e^{-B_{3}s}}$ for some $\mathbf{A_{3}, B_{3}>0}$.}} Proceeding similarly to Step-I,  the geodesic triangle $\Delta \mathcal{A}\mathcal{B}\mathcal{C}$, admits a comparison triangle $\Delta \bar{\mathcal{A}}\bar{\mathcal{B}}\bar{\mathcal{C}}$ in $(\mathbb{D}, \rho_{\mathbb{D}})$. Using a similar argument as in \eqref{E:ae0}, we obtain the following:
\begin{align}\label{E:ae10}
    \kb(\gamma_{2}(2s+T_{s}), {\tau}^{n}_{s}(s+T_{s}))\leq \rho_{\D}(\bar{\gamma}_{2}(2s+T_{s}), \bar{\tau}^{n}_{s}(s+T_{s}))\,\quad \forall s \geq M_{2}.
\end{align}
\noindent
We show that $\rho_{\D}(\bar{\gamma}_{2}(2s+T_{s}), \bar{\tau}^{n}_{s}(s+T_{s}))\leq A'e^{-B's}$ for some $A', B'>0$ and $s\geq M_{2}$. Let $\mathcal{\bar{Q}}^n:=\bar{\tau}^{n}_{s}(b^{n}(s))$ is the foot of the orthogonal projection from the point $\bar{\mathcal{B}}$ on the side $[\bar{\mathcal{A}}, \bar{\mathcal{C}}]$. Similarly to Step-I, the following holds for all $n \in \mathbb{N}$ and $s \geq M_{2}$:

\begin{align}\label{E:ae11}
\rho_{\D}\big(\bar{\mathcal{A}}, \bar{\mathcal{Q}^n}\big)
&\leq \rho_{\D}\big(\bar{\mathcal{A}}, \bar{\mathcal{B}}\big)\notag\\
\rho_{\D}\big(\bar{\tau}^{n}_{s}(0), \bar{\tau}^{n}_{s}(b^{n}(s))\big)
&\leq \kb(\mathcal{A},\mathcal{B})\notag\\
b^{n}(s) &\leq s + T_{s}+ a e^{-bs}.
\end{align}
Analogously to \eqref{E:ae4} in Step-I, we deduce:
\begin{align*}
\rho_{\D}\big(\bar{\mathcal{Q}^n}, \bar{\mathcal{C}}\big)&\leq K_{\bn}(\mathcal{B},\mathcal{C})\notag\\
\rho_{\D}\big(\bar{\mathcal{Q}^n}, \bar{\mathcal{C}}\big)&\leq n-T_{s}+ae^{-2bs-2bT_{s}}\notag\\
\rho_{\D}(\bar{\mathcal{A}}, \bar{\mathcal{C}})-\rho_{\D}(\bar{\mathcal{Q}^n}, \bar{\mathcal{C}})&\geq \rho_{\D}(\bar{\mathcal{A}}, \bar{\mathcal{C}})-n+T_{s}-a e^{-2bs-bT_{s}}\notag\\
\rho_{\D}(\bar{A},\bar{\mathcal{Q}^n})&\geq s+T_{s}-ae^{-bs}-a e^{-2bs-bT_{s}}.
\end{align*}
We will show that $T_{s} \to T_{0}$ as $s \to \infty$ (see Step-III). Therefore,  we conclude from the above equation that there exists $a'>0$ such that 
\begin{align}\label{E:ae12}
b^{n}(s)&\geq s+T_{s}-2a'e^{-bs} \quad \forall s\geq M_{2}.
\end{align}
Combining \eqref{E:ae11}, \eqref{E:ae12} we obtain that 
 \begin{align}\label{E:ae16}
     |b^{n}(s)-(s+T_{s})|&\leq 2a' e^{-bs}.
 \end{align}
\noindent
We choose $A_{1}, M_{1}>0$ large enough in Step-I such that the following relation holds for all $s\geq \max\{M_{1}, M_{2}\}$: 
 \begin{align}\label{E:ae13}
    \rho_{\D}(\bar{\mathcal{B}}, \bar{\tau}^{n}_{s}(b^{n}(s)))\leq A_{1}\sqrt{\big\langle\bar{\mathcal{A}}\mid \bar{\mathcal{C}}\big\rangle_{\bar{\mathcal{B}}}^{\D}}.
\end{align}
\noindent
Proceeding similar to \ref{E:ae7}, we  get an upper bound of $\bigg\langle\bar{\gamma}_{2}(s)\bigm |\bar{\gamma}_{2}(2s+n)\bigg\rangle_{\bar{\gamma}_{2}(2s+T_{s})}^{\D}$ as follows:
\begin{align}\label{E:ae14}
  &\bigg\langle\bar{\gamma}_{2}(s)\mid \bar{\gamma}_{2}(2s+n)\bigg\rangle_{\bar{\gamma}_{2}(2s+T_{s})}^{\D}\notag\\  &\leq \frac{1}{2}\bigg[\rho_{\D}(\bar{\gamma}_{2}(s),\bar{\gamma}_{2}(2s+T_{s}))+\rho_{\D}(\bar{\gamma}_{2}(2s+n),\bar{\gamma}_{2}(2s+T_{s}))-\rho_{\D}(\bar{\gamma}_{2}(s),\bar{\gamma}_{2}(2s+n))\bigg]\notag\\
  &\leq \frac{1}{2}\bigg[s+T_{s}+ae^{-bs}+n-T_{s}+ae^{-2bs-bT_{s}}-(s+n-ae^{-bs})\bigg]\notag\\
  &\leq \frac{\alpha_{0}}{2} \bigg[e^{-bs}+e^{-2bs}+e^{-bs}\bigg]\notag\\
  &\leq \alpha_{0} e^{-bs},
\end{align}
where $\alpha_{0}>0$ depends only on $a$ and $T_{0}=\lim_{s\to\infty}T_{s}$, for all sufficiently large $s>0$.
Therefore, we infer from \eqref{E:ae13}, \eqref{E:ae14}  that 
\begin{align}\label{E:ae15}
\rho_{\D}\big(\bar{\gamma}_{2}(2s+T_{s}),\bar{\tau}^{n}_{s}(b^{n}(s)\big)&\leq a_{1}'e^{-b_{1}'s} \quad \forall s\geq \max\{M_{1},M_{2}\}.
\end{align}

\smallskip

\noindent
where $a_{1}'=A_{1}\sqrt{\alpha_{0}}$ and $b_{1}'=b/2$. We now combine \eqref{E:ae15}, \eqref{E:ae16} and conclude that the following holds for all $n \in \mathbb{N}$ and   $s\geq \max\{M_{1}, M_{2}\}$
\begin{align}\label{E:ae17}
\rho_{\D}\big(\bar{\gamma}_{2}(2s+T_{s}),\bar{\tau}^{n}_{s}(s+T_{s}))&\leq \rho_{\D}(\bar{\gamma}_{2}(2s+T_{s}),\bar{\tau}^{n}_{s}(b^{n}(s))\big)+\rho_{\D}(\bar{\tau}^{n}_{s}(b^{n}(s)),\bar{\tau}^{n}_{s}(s+T_{s})\big)\notag\\
    &\leq  a_{1}'e^{-b_{1}'s}+ 2a'e^{-bs}\leq 2A_{3}e^{-B_{3}s},
\end{align}
where $A_{3}:= \max\{a_{1}', 2a'\}$ and $B_{3}:=\min\{b_{1}',b\}=b/2$.
Hence, \eqref{E:ae10} and \eqref{E:ae17} completes the proof of Step-II. \hfill $\blacktriangleleft$
\medskip

\noindent
From the triangle inequality, we have the following for all $n \in \mathbb{N}$ and  for all $s \geq M_{0}:=\max \{M_{1}, M_{2}\}$
\begin{align*}
    \kb(\gamma_{1}(2s), \gamma_{2}(2s+T_{s}))&\leq  \kb(\gamma_{1}(2s), \sigma^{n}_{s}(s))+\kb(\sigma^{n}_{s}(s),\tau^{n}_{s}(s+T_{s}))\\&+\kb(\tau^{n}_{s}(s+T_{s}), \gamma_{2}(2s+T_{s})) . 
\end{align*}
From the conclusion of step-I and step-II and the above equation
\begin{align}\label{E:ae18}
  \kb(\gamma_{1}(2s), \gamma_{2}(2s+T_{s}))&\leq Ae^{-Bs}+  A_{3}e^{-B_{3}s}+\kb(\sigma^{n}_{s}(s),\tau^{n}_{s}(s+T_{s}))\,\, \forall n \in \mathbb{N}, s \geq M_{0}.
\end{align}
It follows from the Arzelà–Ascoli's theorem that $\tau^{n}_{s}$ has a convergent subsequence. It follwos from the negative curvatre property of the unit ball that $\tau^{n}_{s} \to \tau_{s}$ (upto a subsequence). Therefore, $\tau^{n}_{s}(s+T_{s}) \to \tau_{s}(s+T_{s})$, $\sigma^{n}_{s}(s) \to \sigma_{s}(s) $ as $n \to \infty$ locally uniformly on $[0, \infty)$. Let $\eps_{0}>0$ be any given real number then we choose large $N_{\eps_{0} ,s}\in \mathbb{N}$ such that \begin{align*}
    \kb(\sigma^{n}_{s}(s),\tau^{n}_{s}(s+T_{s}))&<\kb(\sigma_{s}(s),\tau_{s}(s+T_{s}))+\eps_{0} \quad \forall n \geq N_{\eps_{0} ,s}.\notag\\
\intertext{It turns out from \eqref{E:ae9}, and the above equation }
  \kb(\sigma^{n}_{s}(s),\tau^{n}_{s}(s+T_{s}))&<A_{2}e^{-s}+\eps_{0} \quad \forall n \geq N_{\eps_{0} ,s}.
\end{align*}
We now apply the above relation in \eqref{E:ae18}, and conclude that 
\begin{align}\label{E:aae18}
  \kb\big(\gamma_{1}(2s), \gamma_{2}(2s+T_{s})\big)&\leq Ae^{-Bs}+  A_{2}e^{-s}+ A_{3}e^{-B_{3}s}+\eps_{0}.
\end{align}
Since $\eps_{0}>0$ is arbitrary,  
\begin{align}\label{E:ae28}
  \kb(\gamma_{1}(2s), \gamma_{2}(2s+T_{s}))&\leq Ae^{-Bs}+A_{2}e^{-s}+  A_{3}e^{-B_{3}s}\leq 3a_{4}e^{-b_{4}s}\quad \forall s>M_{0},
\end{align}
where $a_{4}=\max\{A, A_{2} A_{3}\}$ and $b_{4}=\min\{b/2, 1\}$. 

\vspace{1.5 em}

\noindent
Next, we show the exponential convergence of $T_{s}$.
\medskip
\\
{\bf Step-III: Showing that there exist  $\mathbf{A_{0}, B_{0}, M_{0}>0}$ and $\mathbf{T_{0}\in \mathbb{R}}$ such that $\mathbf{|T_{s}-T_{0}|<A_{0}e^{-B_{0}s}}$ for all $\mathbf{s \geq M_{0}}$.}  Recall that the horofunction of $\bn$ with the base point $p \in \bn$ at the boundary point $e_{1}$ is defined as follow: 
$$h_{e_{1}}^{p}(z)=\lim_{w \to e_{1}}[\kb(z,w)-\kb(p,w)].$$

At first we show that  there exits $c_{j} \in \mathbb{R}$  and $\alpha_{j}, \beta_{j}>0$ such that $$|\big(h_{e_{1}}^{p}(\gamma_{j}(s))+s\big)-c_{j}|<\alpha_{j}e^{-\beta_{j}s}~~\,\,\forall s\geq M_{0}\quad\forall j \in \{1,2\}.$$

\noindent
Let $\phi_{j}(s)=h_{e_{1}}^{p}(\gamma_{j}(s))+s$ for $j \in \{1,2\}$ and $h>0$.  Then we have the following:
\begin{align}\label{E:ae19}
    &\phi_{j}(s+h)-\phi_{j}(s)\notag\\
    &=h_{e_{1}}^{p}(\gamma_{j}(s+h))+s+h-h_{e_{1}}^{p}(\gamma_{j}(s))-s\notag\\
    &=h_{e_{1}}^{p}(\gamma_{j}(s+h))-h_{e_{1}}^{p}(\gamma_{j}(s))+h \notag\\
    &=\lim_{t \to \infty}[\kb(\gamma_{j}(s+h),\gamma_{j}(t))-\kb(p,\gamma_{j}(t))]-\lim_{t \to \infty}[\kb(\gamma_{j}(s),\gamma_{j}(t))-\kb(p,\gamma_{j}(t))]+h \notag \\
    &=\lim_{t \to \infty}\big[\kb(\gamma_{j}(s+h),\gamma_{j}(t))-\kb(p,\gamma_{j}(t))
    -\kb(\gamma_{j}(s),\gamma_{j}(t))+\kb(p,\gamma_{j}(t))\big]+h \notag\\
    &=\lim_{t \to \infty}\big[\kb(\gamma_{j}(s+h),\gamma_{j}(t))
    -\kb(\gamma_{j}(s),\gamma_{j}(t))\big]+h.
\end{align}
 Since $\gamma_{j}|_{[s,t]}$ is $(1,ae^{-bs})$ quasi-geodesic for all $0\leq s\leq t$, we have the following estimate:
 \begin{align}\label{E:ae20}
   -ae^{-bs-bh}-ae^{-bs} &\leq \kb(\gamma_{j}(s+h),\gamma_{j}(t))
    -\kb(\gamma_{j}(s),\gamma_{j}(t))+h\leq+ae^{-bs-bh}+ae^{-bs}\notag\\ 
  -2ae^{-bs}  &\leq \kb(\gamma_{j}(s+h),\gamma_{j}(t))
    -\kb(\gamma_{j}(s),\gamma_{j}(t))+h \leq 2ae^{-bs}\notag\\
     -2ae^{-bs}  &\leq \lim_{t \to \infty}\big[\kb(\gamma_{j}(s+h),\gamma_{j}(t))
    -\kb(\gamma_{j}(s),\gamma_{j}(t))\big]+h \leq 2ae^{-bs}.
 \end{align}
Combining \eqref{E:ae20}, \eqref{E:ae19} we obtain 
\begin{align}\label{E:ae21}
   | \phi_{j}(s+h)-\phi_{j}(s)|\leq 2ae^{-bs}.
\end{align}
In particular, for all $t>s$ we have that $| \phi_{j}(t)-\phi_{j}(s)|\leq 2ae^{-bs}.$ Therefore, $\sup\{|\phi_{j}(s)|:s \in [0, \infty), j \in \{1,2\}\}<\infty$. It follows from \eqref{E:ae21} that, for any sequence of real numbers 
$\{s_n\}_{n\in\mathbb{N}}$ with $s_{n} \to \infty$, the sequence $\{\phi_j(s_n)\}$ is a Cauchy sequence. 
Consequently, there exists a constant $c_j\in\mathbb{R}$ such that
\[
\phi_j(s_n)\to c_j \qquad \text{as } n\to\infty .
\]
Moreover, \eqref{E:ae21} also implies that $c_j$ is the unique possible 
limit of $\{\phi_j(s_n)\}$ for every sequence 
$\{s_n\}_{n\in\mathbb{N}}$ with $s_n\to\infty$. Hence,
\[
|\phi_j(s)-c_j|\to 0 \qquad \text{as } s\to\infty .
\]
Therefore, given any $\varepsilon>0$, one can choose $t>0$ sufficiently 
large such that
\[
|\phi_j(t)-c_j|<\varepsilon.
\]
Hence, for all $s >0$, taking $t_{\eps}>s>0$ we obtain
\begin{align}
   |\phi_{j}(s)-c_{j}|&\leq  |\phi_{j}(s)-\phi_{j}(t_{\eps})|+|\phi_{j}(t_{\eps})-c_{j}|\notag\\
   &\leq 2ae^{-bs}+\eps.
\end{align}
Since $\eps>0$ in the above equation is arbitrary, it follows that for all $j \in \{1,2\}$
\begin{align}\label{E:ae26}
     |\phi_{j}(s)-c_{j}|&\leq 2ae^{-bs}\quad \forall s\geq 0.
\end{align}
\noindent
We now compute the value of $h_{e_{1}}^{p}(\sigma_{s}(s))$ and  $h_{e_{1}}^{p}(\tau_{s}(s+T_{s}))$.  We have 
\begin{align}\label{E:ae22}
 h_{e_{1}}^{p}(\sigma_{s}(s))&=\lim_{t \to \infty}[\kb(\sigma_{s}(s),\sigma_{s}(t))-\kb(p,\sigma_{s}(t))]\notag\\
 &=\lim_{t \to \infty}[\kb(\sigma_{s}(s),\sigma_{s}(t))-\kb(\sigma_{s}(0),\sigma_{s}(t))+\kb(\sigma_{s}(0),\sigma_{s}(t))-\kb(p,\sigma_{s}(t))]\notag\\
 &=-s+h_{e_{1}}^{p}(\sigma_{s}(0)).
\end{align}
Similarly, 
\begin{align}\label{E:ae23}
 h_{e_{1}}^{p}(\tau_{s}(s+T_{s}))&=  -s-T_{s}+  h_{e_{1}}^{p}(\tau_{s}(0)).
\end{align}
Note that the function $h_{e_{1}}^{p}(z)$ is $1$-Lipschitz function, i.e. $$|h_{e_{1}}^{p}(z_{1})-h_{e_{1}}^{p}(z_{2})|\leq \kb(z_{1},z_{2})\quad \forall z_{1}, z_{2} \in \bn.$$ In particular, 
\begin{align}\label{E:ae24}
    |h_{e_{1}}^{p}(\sigma_{s}(s))-h_{e_{1}}^{p}(\tau_{s}(s+T_{s}))|\leq \kb(\sigma_{s}(s),\tau_{s}(s+T_{s})).
\end{align}
\smallskip

\noindent
Therefore, by substituting the relations obtained from \eqref{E:ae22} and \eqref{E:ae23} into \eqref{E:ae24}, and subsequently applying the inequality in \eqref{E:ae9}, we deduce that

\begin{align}\label{E:ae25}
    |-s+h_{e_{1}}^{p}(\sigma_{s}(0))+s+T_{s}-h_{e_{1}}^{p}(\tau_{s}(0))|&\leq A_{2}e^{-s}\notag\\
    \big|T_{s}-(-h_{e_{1}}^{p}(\gamma_{1}(s))+h_{e_{1}}^{p}(\gamma_{2}(s)))\big|&\leq  A_{2}e^{-s}.
\end{align}

\noindent
We now infer form \eqref{E:ae25} and \eqref{E:ae26} that
\begin{align}\label{E:ae27}
 \big|T_{s}-(c_{2}-c_{1})\big|&\leq   |T_{s}+s-s+h_{e_{1}}^{p}(\gamma_{2}(s))-h_{e_{1}}^{p}(\gamma_{1}(s))+h_{e_{1}}^{p}(\gamma_{1}(s))-h_{e_{1}}^{p}(\gamma_{2}(s))+c_{1}-c_{2}| \notag\\
 &\leq |T_{s}+h_{e_{1}}^{p}(\gamma_{1}(s))-h_{e_{1}}^{p}(\gamma_{2}(s))|+|s+h_{e_{1}}^{p}(\gamma_{2}(s))-c_{2}|+|-s-h_{e_{1}}^{p}(\gamma_{1}(s))+c_{1}|,\notag \\
 \intertext{using \eqref{E:ae25}, \eqref{E:ae26} in the above equation we get the following }
\big|T_{s}-(c_{2}-c_{1})\big| &\leq A_{2}e^{-s}+4ae^{-bs}\leq a_{0}e^{-b_{0}s},
\end{align}
where $a_{0}=\max\{A_{2},4a\}$ and $b_{0}=\min\{1, b\}$. Consequently, it finishes the   step-III. 

\hfill $\blacktriangleleft$

\medskip

\noindent
We infer from the discussion in step-I, step-II, step-III that the following holds for $s>M_{0}>0$ 
\begin{align*}
\kb(\gamma_{1}(2s), \gamma_{2}(2s+(c_{2}-c_{1})))&\leq \kb(\gamma_{1}(2s), \gamma_{2}(2s+T_{s}))+\kb(\gamma_{2}(2s+T_{s}),\gamma_{2}(2s+(c_{2}-c_{1})).\notag\\
\intertext{We now use \eqref{E:ae28},  and  the quasi-geodesic property of $\gamma_{2}$ in  the above relation, and  deduce the  following:}
\kb(\gamma_{1}(2s), \gamma_{2}(2s+(c_{2}-c_{1})))&\leq a_{4}e^{-b_{4}s}+|T_{s}-(c_{2}-c_{1})|+ae^{-b\cdot \min\{2s+T_s,2s+(c_{2}-c_{1})\}}.\notag\\
\intertext{In view of \eqref{E:ae27} and the preceding equation, and assuming $M_{0}>0$ large enough we conclude that there is suitable $a_{1}>0$ such that 
}
\kb\big(\gamma_{1}(2s), \gamma_{2}(2s+(c_{2}-c_{1}))\big)&\leq a_{4}e^{-b_{4}s}+a_{0}e^{-b_{0}s}+a_{1}e^{-2bs}\leq 3\alpha' e^{-\beta' s} \quad \forall s\geq M_{0}.
\end{align*}
where $\alpha'=\max\{a_{4},a_{0},a_{1}\}$ and $\beta'=\min\{b_{4},b_{0},2b\}$=$\min\{1,b/2\}$. Now letting $\alpha=\max\big\{\alpha', \sup\{e^{\beta' s}\kb\big(\gamma_{1}(2s),\gamma_{2}(2s+(c_{2}-c_{1}))\big): s\in [0, M_{0}]\}\big\}$ and   $\beta'=\beta$, we conclude that there exist $T_{0} \in \re$ and $\alpha>0$ such that 
\begin{align*}
  \kb\big(\gamma_{1}(s), \gamma_{2}(s+T_{0})\big)&\leq \alpha e^{-\frac{s}{2}\cdot \min\{1, b/2\}}\quad  \forall s \geq 0.
\end{align*}

\end{proof}
\section{Proof of the Theorem~\ref{T:expcon}}\label{S:expoproof}

We first prove a proposition which will be used in the proof of Theorem~\ref{T:expcon}.
\begin{proposition}\label{P:estimate}
 Let $\Om\St \cn$ be a strongly convex domain with $\smoo^{3}$ boundary and let $\xi \in \partial\Om$. Let    $\phi_{1}, \phi_{2} : \mathbb{D} \to \Omega$ be  two complex geodesics with
$\phi_1(1)=\phi_2(1)=\xi$. Assume that  $
a_j:=\phi_j'(1)$ and $ 
\mu_j:=\langle a_j,\nu\rangle
$ for $ j=1,2$
where $\nu_{\xi}$ is the outward unit normal to $\partial\Omega$ at $\xi$. Define
$$
T:=\frac12\log\frac{\mu_2}{\mu_1},\qquad
\lambda:=e^{-2T}=\frac{\mu_1}{\mu_2},\qquad
V:=a_1-\lambda a_2, \quad \alpha \in (1/2,1].
$$

Then for
$$
s_t:=1-\tanh t,\qquad
z_t:=\phi_1(\tanh t),\qquad
w_t:=\phi_2(\tanh(t+T)),
$$
the following estimates hold as $t \to \infty$:
\begin{equation}\label{eq:st}
s_t=2e^{-2t}+O(e^{-4t}),\qquad 
s_{t+T}=\lambda s_t+O(s_t^2).
\end{equation}

\begin{equation}\label{eq:ztwt}
z_t-w_t=-s_tV+O(s_t^{1+\alpha}).
\end{equation}

\begin{equation}\label{eq:delta}
\delta_{\Omega}(z_t)=\mu_1 s_t+O(s_t^{1+\alpha}),\qquad
\delta_{\Omega}(w_t)=\mu_1 s_t+O(s_t^{1+\alpha}).
\end{equation}

\begin{equation}\label{eq:normal}
\|z_t-w_t\|=O(s_t),\qquad
\left\langle z_t-w_t,\,
\nu_{\pi(z_t)}
\right\rangle
=O(s_t^{1+\alpha}).
\end{equation}
\end{proposition}
\begin{proof}
It follows from Result~\ref{R:Pgeodesic} that $\mu_{j}>0$ for all $j=1,2$. By definition  $$s_{t}-2e^{-2t}=-\frac{2}{e^{2t}(1+e^{2t})}=O(e^{-4t})\quad {\rm as}~~ t \to \infty.$$
From the above relation we obtain 
\begin{align*}
   \frac{ s_{t+T}-\lambda s_{t}}{s_{t}^2}&=\frac{\lambda(1-\lambda)(1+e^{-2t})}{2(1+\lambda e^{-2t})}.
    \end{align*}
Since the right-hand side of the above equation is bounded above as $t \to \infty$, $s_{t+T}=\lambda s_{t}+O(s_{t}^2).$ Therefore, we conclude the relation \eqref{eq:st}.

\smallskip

\noindent
It follows from the Result~\ref{R:Pgeodesic} that the map $\phi_{j}$ has $\smoo^{1, \alpha}$ extension on $\overline{\D}$ for any $\alpha \in (0,1]$ for all $j=1,2.$ Here $\alpha>1/2$. By the Fundamental Theorem of Calculus together with the H\"older continuity of $\phi_j'$, we obtain, for each $j=1,2$,
\begin{align*}
    \phi_{j}(1-h)=\xi-ha_{j}+O(h^{1+\alpha})\quad {\text as} \quad h\to 0^{+}.
\end{align*}
Since $s_{t} \to 0 $ as $t\to \infty$, we get 

\begin{align}\label{E:Tay1}
    \phi_{j}(1-s_{t})=\xi-s_{t}a_{j}+O(s_{t}^{1+\alpha})\quad {\text as} \quad t\to \infty.
\end{align} 
\noindent
We now have the following estimate on $\|z_{t}-w_{t}\|$ as $t \to \infty$:
\begin{align}\label{E:zt-wt}
    z_{t}-w_{t}&=\phi_{1}(\tanh{t})-\phi_{2}(\tanh{(t+T)})\notag\\
    &=\phi_{1}(1-s_{t})-\phi_{2}(1-s_{t+T}).\notag\\
    \intertext{From \eqref{E:Tay1} and above relation we get} 
 z_{t}-w_{t}   &=-s_{t}a_{1}+O(s_{t}^{1+\alpha})+a_{2}s_{t+T}+O(s_{t+T}^{1+\alpha})\notag\\
    &=-s_{t}a_{1}+O(s_{t}^{1+\alpha})+a_{2}(\lambda s_{t}+O(s_{t}^2))+O(s_{t+T}^{1+\alpha})\notag\\
    &=s_{t}(-a_{1}+\lambda a_{2})+O(s_{t}^2)+O(s_{t+T}^{1+\alpha})\notag\\
    &=-Vs_{t}+O(s_{t}^2)+O(s_{t+T}^{1+\alpha})=-Vs_{t}+R_{1}(t),
    \end{align}
where $R_{1}(t)=O(s_{t}^{1+\alpha})$ as $t \to \infty.$ 
This proves the relation \eqref{eq:ztwt}.
Note that Result \ref{R:BalogBonk} ensures the $\smoo^2$ smoothness of  the signed distance function $\rho:\cn \to \re$ defined by 
\begin{equation}
\rho(z)=
\begin{cases}
-\partial_{\Omega}(z), & \text{if } z\in\Omega,\\[2mm]
\phantom{-}\partial_{\Omega}(z), & \text{if } z\in\Omega^{c}
\end{cases}
\end{equation}
in tubular neighborhood of $\Om$. The Taylor expansion of $\rho$ at $\xi$ implies that there exists $\delta>0$ such that, for every $h\in\mathbb{C}^n$ with $\|h\|<\delta$, the following estimate holds:
\begin{align}
    \rho(\xi+h)&=\rl \langle \nabla \rho(\xi), h\rangle +O(\|h\|^2).
\end{align}
Therefore, we get the following as $t \to \infty$
\begin{align}\label{E:bdistzt}
\partial_{\Om}(z_{t})&=-\rho(\xi+z_{t}-\xi)\notag\\
&=-\rl \big\langle \nabla \rho(\xi), (z_{t}-\xi)\big\rangle +O(\|(z_{t}-\xi)\|^2)\\\notag
&=-\rl \big\langle  \nu_{\xi}, (\phi_{1}(\tanh{t})-\xi)\big\rangle +O(\|(z_{t}-\xi)\|^2).\notag\\
\intertext{From \eqref{E:Tay1}, we get }
\partial_{\Om}(z_{t})&=-\rl \big\langle  \nu_{\xi}, (-a_{1}s_{t}+O(s_{t}^{1+\alpha}))\big\rangle +O(\|(z_{t}-\xi)\|^2).\notag\\
&=s_{t}\mu_{1}+O(s_{t}^{1+\alpha})+O(s_{t}^{2})=s_{t}\mu_{1}+O(s_{t}^{1+\alpha}).
\end{align}
Similarly, we deduce that 
\begin{align}\label{E:bdistwt}
    \partial_{\Om}(w_{t})&=-\rl \big\langle  \nu_{\xi}, (-a_{2}s_{t+T}+O(s_{t+T}^{1+\alpha}))\big\rangle +O(\|(w_{t}-\xi)\|^2).\notag\\
&=s_{t+T}\mu_{2}+O(s_{t+T}^{1+\alpha})+O(s_{t+T}^{2}).\notag\\
\intertext{Using \eqref{eq:st} we get }
\partial_{\Om}(w_{t})&=\lambda s_{t}\mu_{2}+O(s_{t}^2)+O(s_{t+T}^{1+\alpha})+O(s_{t+T}^{2})=\mu_{1}s_{t}+O(s_{t}^{1+\alpha}).
\end{align}
Hence,  \eqref{E:bdistzt}, \eqref{E:bdistwt} proves the relation \eqref{eq:delta}.

\smallskip

\noindent
Let $N_{\eps}=\cup_{p \in \partial \Om}\{p+t\nu_{p}: t \in (-\eps, \eps)\}$.
It is a fact that there exists $\eps>0$ such that for all $z \in N_{\eps}$ there exists a unique point $\pi(z) \in \partial\Om$ such that $\delta_{\Om}(z)=\|z-\pi(z)\|$. It follows from Result~\ref{R:BalogBonk} that the map $\pi: N_{\eps}\to  \cn$ is a $\smoo^{1}$ map. Since the domain $\Om$ has $\smoo^{3}$ boundary, the map $z \mapsto \nu_{z}$ is a $\smoo^{1}$ smooth map on $\Om$. Consequently, $z\mapsto \nu_{\pi(z)}$ is $\smoo^{1}$ smooth on $N_{\eps}$. Note that $z_{t} \to \xi$ as $t\to \infty$. Hence, we get that
\begin{align}
   \|\nu_{\pi(z_{{t}})}-\nu_{\xi}\| <C\|z_{t}-\xi\|.
\end{align}
It follows from \eqref{E:Tay1} that $ \|\nu_{\pi(z_{{t}})}-\nu_{\xi}\|=O(s_{t})$. Therefore,  $\nu_{\pi(z_{t})}=\nu_{\xi}+R(t)$ where $R(t)=O(s_{t})$. We now deduce the following:
\begin{align}\label{E:innerpro}
    \left\langle z_t-w_t,\,
\nu_{\pi(z_t)}
\right\rangle&=\left\langle z_t-w_t,
\nu_{\xi}+R(t)
\right\rangle \notag\\
&=\big\langle z_t-w_t,
\nu_{\xi}
\big\rangle+\left\langle z_t-w_t,
R(t)
\right\rangle.\notag\\
\intertext{Using \eqref{E:zt-wt} in the above relation we get that}
\left\langle z_t-w_t,\,
\nu_{\pi(z_t)}
\right\rangle&=\big\langle -Vs_{t}+R_{1}(t),
\nu_{\xi}
\big\rangle+\langle z_t-w_t,
R(t)\rangle\notag\\
&=-s_{t}\big\langle V, \nu_{\xi}\big\rangle+\big\langle R_{1}(t), \nu_{\xi}\big\rangle+\langle z_t-w_t,
R(t)\rangle \notag\\
&=-s_{t}\big\langle a_{1}-\lambda a_{2}, \nu_{\xi}\big\rangle+\big\langle R_{1}(t), \nu_{\xi}\big\rangle+\langle z_t-w_t,
R(t)\rangle,
\end{align}
$\|R_{1}(t)\|=O(s_{t}^{1+\alpha})$ and $\|R(t)\|=O(s_{t})$ as $t \to \infty$.

\smallskip

\noindent
Note that it follows from the hypothesis that $\big\langle a_{1}-\lambda a_{2}, \nu_{\xi}\big\rangle=\mu_{1}-\lambda \mu_{2}=0$. Now from Cauchy-Schwarz inequality and \eqref{E:innerpro} we conclude that  $\left\langle z_t-w_t,\,
\nu_{\pi(z_t)}
\right\rangle=O(s_{t}^{1+\alpha})$. This proves \eqref{eq:normal}. This completes the proof of the proposition.

\end{proof}

\begin{proof}[Proof of Theorem~\ref{T:expcon}]
Let $\gamma_{1}, \gamma_{2}: [0, \infty) \to \Om$ be two geodesic rays such that $\gamma_{1}(\infty)=\gamma_{2}(\infty)=\xi \in \partial \Om.$ It follows from Result~\ref{R:Pgeodesic} that every real geodesic rays lies in a complex geodesic. Let $\phi_{1}, \phi_{2}:\D \to  \Om$ be two complex geodesics such that $\gamma_{j}[0, \infty) \st \phi_{j}(\D)$  for $j=1,2$. We first assume that the $\phi_{1}, \phi_{2}$ are two distinct complex geodesics. After composing with a suitable automorphism of $\D$ we can assume that $\phi_{j}(0)=\gamma_{j}(0)$ for $j=1,2$. Now, composing with a suitable rotation, we can assume that $\phi_{j}(1)=\gamma_{j}(\infty)=\xi$. If we assume that $\gamma_{j}(t)=\phi_{j}(s_{j}(t))$, then it turns out that $s_{j}:[0,\infty) \to \D$ is a geodesic ray with respect to the Poincar\'e distance with $s_{j}(0)=0$ and $s_{j}(\infty)=1$ for $j=\{1,2\}$. Consequently $s_{j}(t)=\tanh{t}$.  We are now in a position to invoke Proposition~\ref{P:estimate}. Associated with each complex geodesic $\phi_j$, $j=1,2$, we consider the
constants $a_j$ and $T$ provided by Proposition~\ref{P:estimate}. We now have  $\gamma_{1}(t)=\phi_{1}(\tanh{t})$ and $\gamma_{2}(t+T)=\phi_{2}(\tanh{(t+T)})$. We use Proposition~\ref{P:estimate}  taking $z_{t}:=\gamma_{1}(t)$  and $w_{t}:=\gamma_{2}(t+T)$.  In view of Result~\ref{R:nikolov-thomas}, we conclude that there exists $C>0$ such that 
\begin{align}\label{E:kdist}
    K_{\Om}(z_{t},w_{t})\leq \ln{\bigg(1+C.A_{\Om}(z_{t},w_{t})\bigg)},
\end{align}
where $$A(z_{t},w_{t})= \dfrac{|\big\langle (z_{t}-w_{t}), \nu_{\pi(z_{t})}\big \rangle|+\|z_{t}-w_{t}\|^2+\|z_{t}-w_{t}\|\sqrt{\partial_{\Om}(z_{t})}}{\sqrt{\partial_{\Om}(z_{t})\partial_{\Om}(w_{t})}}.$$ 
It  follows from  \eqref{eq:delta} that there exist   $M>-T>0$ and $c,C>0$ such that for all $t>M$ the following holds: 
\begin{align}\label{E:dbesti}
cs_{t}\leq\partial_{\Om}(z_{t})\leq Cs_{t},\notag\\
cs_{t}\leq\partial_{\Om}(w_{t})\leq Cs_{t}.
\end{align}
\noindent
Now  from \eqref{eq:normal} it follows that we can choose $C, M>0$ suitably large such that $$\big|\big\langle (z_{t}-w_{t}), \nu_{\pi(z_{t})}\big \rangle\big|<Cs_{t}^{1+\alpha} \quad \forall t>M.$$  From \eqref{eq:normal} we gat that $\|z_{t}-w_{t}\|<Cs_{t}$ for all $t\geq M$. Therefore, invoking Proposition~\ref{P:estimate} and using  \eqref{E:dbesti} we infer that there exists $C_{1}>0$ such that  for all $t>M$ we have 
\begin{align}\label{E:kdist1}
   A(z_{t},w_{t})\leq \dfrac{Cs_{t}^{1+\alpha}+C^2s_{t}^2+Cs_{t}^{3/2}}{cs_{t}}<C_{1}s_{t}^{1/2}. 
\end{align}
Since $\ln{(1+x)}\leq x$ for all $x \geq 0$ and $\alpha>0$, we infer from  \eqref{E:kdist},~~\eqref{E:kdist1}  that   
$$K_{\Om}(z_{t}, w_{t})\leq C.C_{1}s_{t}^{1/2}\quad\forall t\geq M.$$  Hence, \eqref{eq:st}  implies   $K_{\Om}(\gamma_{1}(t), \gamma_{2}(t+T))\leq \tilde{C}e^{-t}$ for all $t\geq M$  for some  $\tilde{C}>0$. Taking the constant $\tilde{C}>0$ large enough we conclude that \begin{align}\label{E:fin0}
    K_{\Om}(\gamma_{1}(t), \gamma_{2}(t+T))\leq \tilde{C}e^{-t} \quad \forall  t\geq 0.
\end{align} 

\smallskip

\noindent
We now prove the lower estimate of  $K_{\Om}(\gamma_{1}(t), \gamma_{2}(t+T))$. It follows from the expression of $A(z_{t},w_{t})$ that 
\begin{align}\label{E:estim2}
A(z_{t},w_{t})&\geq \frac{\|z_{t}-w_{t}\|}{\sqrt{\partial_{\Om}(w_{t})}}
\end{align}
Note that $V=0$ implies that $a_{1}=\lambda a_{2}$. Then Result~\ref{R:Pgeodesic} implies that the two geodesics $\phi_{1}$ and $\phi_{2}$ are same. Since we have assumed that the geodesics $\phi_{1}$ and $\phi_{2}$ are distinct,  $V \neq 0 $. Note that  \eqref{E:zt-wt} allows us to choose $M>0$ large enough such that 
\begin{align*}
    \frac{\|z_{t}-w_{t}\|}{s_{t}}>\frac{V}{2}\quad \forall t>M.
\end{align*}
Combining \eqref{E:estim2} and the above equation, we infer that 
\begin{align}\label{E:estm3}
A(z_{t},w_{t})\geq \frac{V}{2}s_{t}^{1/2} \quad\forall t>M.
\end{align}
\noindent

\noindent
It follows from the upper bound of $A(z_{t},w_{t})\to 0$ as $t \to \infty$. It is a fact that  $\ln{(1+x)}\geq x/2$ for all $0<x<1$. 
Taking $M>0$ large enough  and appealing to Result~\ref{R:nikolov-thomas} we conclude that 
\begin{align}
    K_{\Om}(z_{t},w_{t})\geq \ln{\big(1+cA(z_{t},w_{t})\big)}\geq c\cdot\frac{A(z_{t},w_{t})}{2}\quad \forall t>M.
\end{align}
\noindent
We now use \eqref{E:estm3},  and  the above equation and conclude that there exists $c_{1}>0$ such that

\begin{align}\label{E:fin1}
   K_{\Om}(z_{t},w_{t})&\geq c_{1}s_{t}^{1/2}\geq c_{1}e^{-t}.
\end{align}

\smallskip

\noindent
Therefore, \eqref{E:fin0}  and \eqref{E:fin1} yields that $$\lim_{t \to \infty}\frac{1}{t}\ln{K_{\Om}(\gamma_{1}(t), \gamma_{2}(t+T))}=-1.$$

\medskip

\noindent
We now assume that both the geodesic rays  $\gamma_{1}, \gamma_{2}$ are in the image of same complex geodesic. Let $\phi:\D \to \Om$ be a complex geodesic  such that $\gamma_{j}[0,\infty)\st \phi(\D)$. If $\gamma_{j}(t)=\phi(s_{j}(t))$. Then it turns out that
$s_{j}:[0,\infty) \to \D$ are two geodesic rays in $\D$ with respect to the Poincar\'e distance for $j=1,2$ with $s_{1}(\infty)=s_{2}(\infty)$.  Now it follows from Result~\ref{R:expcball} that there exists $T_{0}\in \mathbb{R}$ such that \begin{align}\label{E:expoball1}
    \lim_{t \to \infty}\frac{1}{t}\ln{K_{\D}\big(s_{1}(t), s_{2}(t+T_{0})\big)}=-2.
\end{align}

\noindent
It is a fact that for the complex geodesic $\phi$
 there exists a  map $p:\Om\to \D$ such that $p\circ \phi =i_{\D}.$ From this it follows that $K_{\Om}(\phi(x), \phi(y))=K_{\phi(\D)}(\phi(x), \phi(y))$ for  all $x,y \in \D$. Therefore, we conclude that $$K_{\Om}\big(\phi(s_{1}(t)),\phi(s_{2}(t+T_{0}))\big)=K_{\phi(\D)}\big(\phi(s_{1}(t)),\phi(s_{2}(t+T_{0}))\big)=K_{\D}(s_{1}(t),s_{2}(t+T_{0})).$$
We now use the above relation in \eqref{E:expoball1} and  obtain that 

 $$\lim_{t \to \infty}\frac{1}{t}\ln{K_{\Om}\big(\gamma_{1}(t), \gamma_{2}(t+T_{0})\big)}=-2 .$$

Hence, we  conclude  the theorem.
 
\end{proof}
\section{Proof of Theorem~\ref{T:Application} and Theorem~\ref{T:c(f)}}\label{S:appli}

The proof of Theorem~\ref{T:Application} follows the general framework of the proof given \cite[Theorem~5.5]{Zimmer2018}. We therefore restrict ourselves to a sketch of the argument, focusing on the features that are specific to the present setting.  A key ingredient in the proof is Theorem~\ref{T:expcon}, whose application plays a crucial role in establishing the result. We begin by establishing the following lemma, which is analogous to \cite[Proposition~2.1]{Zimmer2018}.

\begin{lemma}\label{L:AppLemma}
 Let $d \geq 2$ and $\{e_{1}, e_{2}, \cdots, e_{d}\}$ denotes the standard basis of $\cn$. Suppose that  $\mcd \st \cn$ is a convex domain with the following property:
 \begin{itemize}
     \item [(i.)] $\mcd \cap {\rm Span}_{\cplx}\{e_{2}, \cdots ,e_{d}\}=\emptyset$,
     \item [(ii.)] $\mcd \cap {\rm Span}_{\cplx}\{e_{1}\}=\{ze_{1}: {\rm Im}(z)>0\}$,
     \item[(iii.)]
     $\mcd$ is biholomorphic to a bounded,  strongly convex domain $\mathcal{G} $ with $\smoo^3$ boundary in $\cn$.
 \end{itemize}
 Then \begin{equation*}
     \liminf_{r \to \infty}\frac{1}{r}\ln{\partial_{\mcd}(ie^{r}e_{1};v)}\geq \frac{1}{2}\quad~\forall v \in {\rm span}_{\cplx}\{e_{2}, e_{3}, \cdots ,e_{d}\}.
 \end{equation*}
\end{lemma}
\begin{proof}
By similar observation as in the proof of \cite[Proposition~2.1]{Zimmer2018} (see \cite[Observation~2.12]{Zimmer2018}) we have $\forall v \in {\rm Span}_{\cplx}\{e_{2}, \cdots ,e_{d}\}$ there exists $\alpha_{v} \in [0,\infty]$ such that 
\begin{align*}
 \{ze_{1}+v: {\rm Im}(z)>\alpha_{v}\}=\mcd \cap (v+\cplx e_{1}).   
\end{align*}
Let $\mathcal{V}$ be the set of all unit vectors in ${\rm Span}\{e_{2}, \cdots ,e_{d}\}$. Note that $ie_{1}\in\mathcal{D}$ and $\mathcal{D}$ is open, hence, there exists $r>0$ such that $B(ie_{1},r)\subset\mathcal{D}$. Choose $\delta>0$ such that $2\delta<r$. Then, 
\begin{align*}
    ie_{1}+2\delta \D.v \subset \mcd\quad \forall v \in \mathcal{V}.
\end{align*}

We now consider the  curves $\gamma, \gamma_{v}:[0, \infty) \to \cn $ defined as follows
\begin{align*}
    \gamma(t)&=ie^{2t}e_{1}\notag\\
  \gamma_{v}(t)&=\delta v+ (\alpha_{\delta v}+e^{2t})ie_{1}, \quad~{\rm where }~~v\in \mathcal{V}. 
 \end{align*}
Then arguing similar as in the proof of \cite[Proposition~2.1]{Zimmer2018} we conclude that the curves $\gamma,\gamma_{v}$ are geodesic rays in $(\mcd, K_{\mcd})$ with $\lim_{t \to \infty}K_{\mcd}(\gamma(t), \gamma_{v}(t))=0$. Therefore, 
if $\psi:\mcd \to \mathcal{G}$ is a biholomorphism, then $\psi \circ \gamma$ and $\psi \circ \gamma_{v}$ are two strongly asymptotic  geodesic ray in $\mathcal{G}$. Note that the Busemann function  is $1$-Lipschitz. Therefore, the condition $K_{\mathcal{G}}(\psi \circ \gamma(t),\psi \circ \gamma_{v}(t)) \to 0$ as $t \to \infty$ implies  that  $\psi(\gamma(0))$ and $\psi(\gamma_{v}(0))$ are on same horosphere centered at $\psi\circ \gamma(\infty)=\psi\circ \gamma_{v}(\infty) \in \partial \mathcal{G}$. We now invoke Theorem~\ref{T:expcon} to conclude that there exists $A_{v}>0$ such that 
\begin{align}\label{E:appe0}
K_{\mcd}(\gamma(t), \gamma_{v}(t))=K_{\mathcal{G}}(\psi \circ\gamma(t),\psi\circ\gamma_{v}(t))<A_{v}e^{-t}\quad \forall t \geq 0.
\end{align}
Let us define $s_{t,v}=t+\frac{1}{2}\ln{(1-\frac{\alpha_{\delta v}}{e^{2t}})}$ for suitable large $t$. Then proceeding similar as in the proof of \cite[Proposition~2.1]{Zimmer2018} we conclude that 
\begin{align}\label{E:appe1}
   | K_{\mcd}(\gamma(t), \gamma_{v}(t))-K_{\mcd}(\gamma(t), \gamma_{v}(s_{t,v}))|&\leq \frac{1}{2}|\ln{(1-\frac{\alpha_{\delta v}}{e^{2t}})}| \notag\\
&=\frac{\alpha_{\delta v}}{2}e^{-2t}+O(e^{-4t}).
\end{align}
Therefore, combining \eqref{E:appe0} and \eqref{E:appe1}
 we derive the following for every $v \in \mathcal{V}$
 \begin{align}\label{E:Eapp1}
     K_{\mcd}(\gamma(t), \gamma_{v}(s_{t,v}))&\leq e^{-t}\big(A_{v}+\frac{\alpha_{\delta v}}{2}e^{-t}+O(e^{-2t})\big)\notag\\
     \frac{1}{t}\ln{K_{\mcd}(\gamma(t), \gamma_{v}(s_{t,v}))}&\leq -1+\frac{1}{t}\ln{\big(A_{v}+O(e^{-t})\big)} \notag\\
     \limsup_{t \to \infty}\frac{1}{t}\ln{K_{\mcd}(\gamma(t), \gamma_{v}(s_{t,v}))}&\leq -1. 
   \end{align}

Let $t_{n} \to \infty$ be  a sequence of real numbers such that $\liminf_{t \to \infty}\frac{1}{t}\ln{\partial_{\mcd}(e^{t}ie_{1};v)}=\lim_{n \to \infty} \frac{1}{2t_{n}}\ln{{\partial_{\mcd}(e^{2t_{n}}ie_{1};v)}}
$. Following the similar computation as \cite[Proposition~2.1]{Zimmer2018}
we conclude that there exist $C>0$ and $v_{0}\in \mathcal{V}$ such that 
\begin{align*}
    K_{\mcd}(\gamma(t_{n}),\gamma_{v_{0}}(s_{t_{n},v_{0}}))&\geq \frac{C}{\partial_{\mcd}(e^{2t_{n}}ie_{1};v)} \notag\\
    \frac{1}{t_{n}}\ln{K_{\mcd}(\gamma(t_{n}),\gamma_{v_{0}}(s_{t_{n},v_{0}}))}&\geq -\frac{1}{t_{n}}\ln{\partial_{\mcd}(e^{2t_{n}}ie_{1};v)}+\frac{\ln{C}}{t_{n}}\notag\\
\intertext{taking  $n \to \infty$ in the above equation, we infer from \ref{E:Eapp1} that }
-1&\geq -2\liminf_{t \to \infty}\frac{1}{t}\ln{\partial_{\mcd}(e^{t}ie_{1};v)}.
\end{align*}
Consequently, $\frac{1}{2}\leq \liminf_{t \to \infty}\frac{1}{t}\ln{\partial_{\mcd}(e^{t}ie_{1};v)}~~\forall v\in \mathcal{V}$. It follows form the definition that the inequality is true for all $v \in {\rm span}\{e_{2}, e_{3}, \cdots ,e_{d}\}$. This finishes the proof of the Lemma.
\end{proof}
We now give the proof of Theorem~\ref{T:Application} 
\begin{proof}[Proof of Theorem~\ref{T:Application} ]
Let $f:\mathbb{X}_{d,0} \to \mathbb{R}$ be a continuous intrinsic function satisfying the hypothesis of the theorem. To obtain a contradiction, suppose that the conclusion of the theorem is false. Then it follows that there exists a sequence $\mcd_{n} \in \mathbb{X}_{d,0}$ such that:
\begin{itemize}
    \item [(a.)]
  For all $n \in \mathbb{N}$,  $\mcd_{n}$ has $\smoo^{2, \alpha}$ boundary for some $\alpha>0$,
  \smallskip
    \item [(b.)]
    For all $n \in \mathbb{N}$,  $\mcd_{n}$ is not strongly pseudoconvex,
    \smallskip
   \item [(c.)]
   $\big|f(\mcd_{n},z)-f(\Om,0)\big|\leq \frac{1}{n}$, outside some compact subset of $\mcd_{n}$.
\end{itemize}

\smallskip

\noindent
Given some $\mcd \in \mathbb{X}_{d}$, let ${\rm BlowUp}(\mcd) \subset \mathbb{X}_{d}$ denotes the set of all $\mcd_{\infty}$ in $\mathbb{X}_{d}$ such that there exists a sequence  $x_{n} \in \mcd$, $x_{\infty} \in \mcd_{\infty}$ and affine map $A_{n} \in {\rm Aff}(\cn)$ with the following property:
\begin{itemize}
    \item
 For every compact subset $K \st \mcd$,   there exists $N\in \mathbb{N}$ such that $x_{n} \notin K$ for all $n >N$.
 \item 
 $\big(A_{n}\mcd, A_{n}x_{n}\big)$ converges to $(\mcd_{\infty}, x_{\infty})$ in $\mathbb{X}_{d,0}$.
    \end{itemize}
    \noindent
  We now infer from  \cite[Proposition~5.8]{Zimmer2018}  that for each $n \in \mathbb{N}$, there exists $\mcd_{n, \infty} \in {\rm BlowUp}(\mcd_{n})$ such that the following holds:
 \begin{itemize}
     \item [(1.)]
     $\mcd_{n, \infty} \in ie_{1}+\mathbb{K}_{d}$,
     \item[(2.)]
     $\mcd_{n, \infty} \cap {\rm Span}_{\cplx}\{e_{2},e_{3},\cdots ,e_{d}\}=\emptyset$
     \item[(3.)] 
     $\mcd_{n, \infty} \cap \cplx e_{1}=\{ze_{1}: {\rm Im}(z_{1})>0\}$,
     \item[(4.)] 
     $\partial_{\mcd_{n,\infty}}(e^{r}ie_{1};e_{2})\leq e^{\frac{r}{2+\alpha}}$ for $\alpha>2$ and $r\geq 1.$
 \end{itemize}
 Here $\mathbb{K}_{d} \st \mathbb{X}_{d}$ be the set of convex domains $D$  such that
 \begin{itemize}
     \item 
     $D_{1}e_{i} \st \Om$ for each $1\leq i \leq d,$
     \smallskip
     \item 
     $Z_{i} \st \Om=\emptyset$ for each $1\leq i \leq d$
 \end{itemize}
where $D_{1}=\{z \in \cplx: |{\rm Im}(z)|+|{\rm Re}(z)|<1\}$ for  $1\leq i\leq d$  and  $Z_{i}=e_{i}+{\rm Span}_{\cplx}\{e_{i+1}\cdots,e_{d}\}$. It is proved in  \cite{Franknel91}  that $\mathbb{K}_{d,0}=\{(D,0):D \in \mathbb{K}\}$ is compact. Now proceeding similarly as in the proof of \cite[Theorem 5.5]{Zimmer2018} we get that
\begin{align}
    \big|f(\mcd_{n, \infty},z)-f(\Om,0) \big|\leq \frac{1}{n}\quad \forall n \in \mathbb{N}.
    \end{align}
Since $\mathbb{K}_{d}\st \mathbb{X}_{d}$ is a compact subset, after passing to a subsequence,  we get  $\mcd_{n, \infty} \to \mcd_{\infty}$ in $\mathbb{X}_{d}$. Consequently, the continuity of the function $f$ implies that
\begin{align*}
    f(\mcd_{\infty},z)=f(\Om,0)~\quad \forall z \in \mcd_{\infty}.
\end{align*}
Therefore, from the hypothesis of the theorem we get that $\mcd_{\infty}$ is biholomorphic to the strongly convex domain $\Om$.  Since $\mcd_{n, \infty} \to \mcd_{\infty} $ with respect to the local Hausdorff distance, it follows  form the property $(2), (3), (4)$ of $\mcd_{n, \infty}$ and from the local Hausdorff convergence that the limit  domain $\mcd_{\infty}$ satisfies the following condition:
\begin{itemize}
    \item [(i.)]
    $\mcd_{\infty} \cap {\rm Span}_{\cplx}\{e_{2},\cdots, e_{d}\}=\emptyset$,
    \item[(ii.)]
    $\mcd_{\infty} \cap \cplx e_{1}=\{ze_{1}:{\rm Im}(z)>0\}$,
    \item [(iii.)]
    $\partial_{\mcd_{\infty}}(e^{r}ie_{1};e_{2}) \leq e^{\frac{r}{2+\alpha}}$ for $r\geq 1$.
\end{itemize}
Therefore, the domain $\mcd_{\infty}$ satisfies the conditions of Lemma~\ref{L:AppLemma}. We see that  the property $(\rm iii.)$ of $\mcd_{\infty}$ obtained as above contradicts  the conclusion of  Lemma~\ref{L:AppLemma}. 
\begin{align*}
    \frac{1}{2} \leq\liminf_{r \to \infty} \frac{1}{r}\ln{\partial_{\mcd_{\infty}}(e^{r}ie_{1};e_{2})} \leq \frac{1}{r}\ln{e^{\frac{r}{2+\alpha}}}=\frac{1}{2+\alpha}.
\end{align*}
This implies that $\alpha\leq 0$, which contradicts the hypothesis.
\end{proof}

The proof of Theorem~\ref{T:app2} proceeds along the same lines as the proof of Theorem~\ref{T:Application}.

We now give the proof of Theorem~\ref{T:c(f)}.
\begin{proof}[Proof of Theorem~\ref{T:c(f)} ]
We first show that $f\circ\gamma $ is $(1,Ae^{-bs})$ quasi-geodesic on $[s,\infty)$ for $s\geq 0$. Clearly, $\kb(f(\gamma(s)), f(\gamma(t)))\leq |s-t|$ for all $s,t\geq 0$. It follows from our hypothesis  and  the reverse triangle inequality that  for $v>u \geq s\geq 0$ :
\begin{align}\label{E:ff1}
\big|\kb(\gamma(0),f(\gamma(u)))-\kb(\gamma(0),f(\gamma(v)))\big|&\leq \kb(f(\gamma(u)),f(\gamma(v)))\notag\\
\kb(p,f(\gamma(u)))-u+u-v+v-\lambda+\lambda-\kb(p,f(\gamma(v)))&\leq \kb(f(\gamma(u)),f(\gamma(v)))\notag\\
(v-u)-Ae^{-bv}-Ae^{-bu}&\leq \kb(f(\gamma(t)),f(\gamma(s)))\notag\\
(v-u)-2Ae^{-bu}&\leq \kb(f(\gamma(u)),f(\gamma(v))).
\end{align}
Switching the role of $u,v$ in the above equation, we obtain that
$$|u-v|-2Ae^{-bs}\leq \kb(f(\gamma(u)),f(\gamma(v)))\leq |u-v|\quad \forall u,v \geq s\geq 0.$$
If $\sigma:[0, \infty) \to \bn$ is another geodesic ray with $\sigma(\infty)=\xi$, then there exist  $T\in \mathbb{R}$, $A_{0}>0$ such that $$ \kb(\sigma(t+T), \gamma(t))\leq A_{0}e^{-t}\quad \forall t \geq 0.$$

\begin{align*}
    \big|\kb(\gamma(0), f(\sigma(t+T)))-t-\lambda\big|&\leq\big|\kb(\gamma(0), f(\sigma(t+T)))-\kb(\gamma(0), f(\gamma(t)))\big|\notag\\
&+\big|\kb(\gamma(0), f(\gamma(t)))-t-\lambda\big|\notag\\
&\leq \kb(f(\sigma(t+T)),f(\gamma(t)))+Ae^{-bt}\notag\\
&\leq \kb(\sigma(t+T), \gamma(t))+Ae^{-bt}.
\end{align*}
Now combining the above equations we get that
\begin{align*}
    \big|\kb(\gamma(0), f(\sigma(t+T)))-t-\lambda\big|
    &\leq 2 \max\{A,A_{0}\}e^{-t(\min\{1,b\})}. 
\end{align*}
Therefore, it follows from \eqref{E:ff1} that 
\begin{align*}
   |u-v|-2Ae^{-bu}&\leq \kb(f(\sigma(u)),f(\sigma(v)))\leq |u-v|\quad \forall u,v>s>\max\{0,-T\}. 
\end{align*}
Consequently, for any geodesic ray $\sigma$ in $\bn$ with $\sigma(\infty)=\xi$, there exist  $M_{\sigma}, b'>0$ such that the  map $f\circ \sigma\big|_{[s,\infty)}$ is a $(1,e^{-b's})$ quasi-geodesic in $\bn$ for all $s>M_{\sigma}>0$. Therefore, invoking Theorem~\ref{P:exponetial} we conclude that  there exist $T_{\sigma}\in \mathbb{R}$ and $a, b, M_{\gamma}>0$   such that 
\begin{align*}
\kb(f(\sigma(s)), \sigma(s+T_{\sigma}))<ae^{-bs}\quad \forall s>M_{\gamma}.    
\end{align*}
From the above equation, we obtain that  \begin{align*}
    \lim_{t \to \infty}\inf_{s \geq 0}\kb(f(\sigma(t),\sigma(s)) )=0.
\end{align*}
Now \cite[Proposition~4.10]{AF2024} yields that 
$\lim_{t\to \infty}\kb(f(\sigma(t)),\sigma(t-\ln{\lambda}))=0.$

\end{proof}
\noindent
We conclude this section with an example satisfying the condition of Theorem~\ref{T:c(f)}.

\begin{example}\label{Examp:E1}
   Let $\mathcal{S} = \{ (z, w) \in \mathbb{C}^2 : {\rm Im}(z) > |w|^2 \}$ and $\gamma:[0, \infty) \to \mcs$ defined by  $\gamma(t)=(ie^{t},0)$. Let   $F: \mathcal{S} \to \mathcal{S}$  defined by \[ F(z, w) = \left( \lambda z - \frac{1}{z}, \frac{w}{z+i} \right), \quad {\rm where}\, \lambda>1. \] 

\smallskip
\noindent
{\bf Explanation:} Note that ${\rm Im}(\lambda z-1/z) = {\rm Im}(z)\left(\lambda + \frac{1}{|z|^2}\right)$ and ${\rm Im}(z)>0$  implies $|z+i|^2 > |z|^2.$ Therefore,
\begin{align*}
   {\rm Im}(z)\left(\lambda + \frac{1}{|z|^2}\right)&\geq |w|^2\left(\lambda + \frac{1}{|z|^2}\right)>\frac{|w|^2}{|z+i|^2}.
\end{align*}
Hence, $F(\mcs) \subset \mcs.$ It is clear that $\gamma$ is a geodesic ray in $\mcs$ and $F$ has no fixed point in $\mcs$. Along the geodesic, $F(\gamma(t)) = (i(\lambda e^t + e^{-t}), 0)$. Since the image lies in the totally geodesic slice $w=0$, the Kobayashi distance reduces to the one-dimensional case:
\[ K_{\mcs}((i,0), F(\gamma(t))) = K_{\mathbb{H}}(i, i(\lambda e^{t}+e^{-t}))=\ln(\lambda e^t + e^{-t}) = t + \ln \lambda + \ln\left(1 + \frac{1}{\lambda}e^{-2t}\right). \]
This yields the same quantitative defect:
\[ |K_{\mcs}((i,0), F(\gamma(t))) - t - \ln \lambda| \le \frac{1}{\lambda}e^{-2t}. \]
Thus, the condition is satisfied with $A = 1/\lambda$, and $b = 2$.
It can also be verified by explicit formula of $F$ and the Kobayashi distance   that the hypothesis $(i)$ of the theorem is also satisfied. 
\end{example}
\section{Exponential Convergence: Local and Global Aspects}\label{S:expo local to global}

We begin this section with the following proposition that will be used in the proof of Theorem~\ref{T:Glob to loc}.

\begin{proposition}\label{L:range of geodesic}
 Let $\xi \in \partial \bn$. Then for all 
$\eps>0$ there exists $\delta_{\eps}>0$ such that  if  $p,q \in \bn \cap B(\xi, \delta_{\eps})$, and 
$\sigma_{pq}:[0, \kb(p,q)] \to \bn$ be the  $\bn$ geodesic joining the points $p,q$. Then $\sigma_{pq}([0, \kb(p,q)]) \subset \bn \cap B(\xi, \eps)$.
\end{proposition}

\begin{proof}
 We show that $\forall \eps>0$ there exists $\delta_{\eps}>0$ such that $\forall p, q \in B(e_{1}, \delta_{\eps})$ the image of the $\bn$-geodesic $\sigma_{p,q}$  is contained in $B(e_{1}, \eps) \cap \bn$. Here $e_{1}=(1,0')$

\noindent
Let $\varepsilon>0$ be given. By Gehring-Hayman  inequality (see \cite[Theorem~1]{KNT2024}) we get
$C\geq 1$ such that for all $a,b \in \bn$
\[
l_E(\sigma_{ab}) \leq C|a-b|,
\]
where $\sigma_{ab}$ is the arc-length parametrized $\bn$-geodesic joining $a,b$ and $l_{E}(\sigma_{ab})$ denotes the Euclidean length of the curve $\sigma_{ab}$.
\noindent
Let
\[
0<\delta_{\varepsilon}<\frac{\varepsilon}{2C+1}.
\]

If $p,q\in B(e_{1},\delta_{\varepsilon})\cap \bn$ and
\[
\sigma_{pq}:[0,\kb(p,q)]\to \bn
\]
is the $\bn$-geodesic joining $p,q$, then we have the following for all
$t\in [0,\kb(p,q)]$:

\begin{align*}
|\sigma_{pq}(t)-1|
&\leq
|\sigma_{pq}(t)-p|+|p-1|\leq
l_{\mathrm{Euc}}\bigl(\gamma_{pq}|_{[0,t]}\bigr)+\delta_{\varepsilon}\leq l_{\mathrm{Euc}}(\sigma_{pq})+\delta_{\varepsilon}.\notag\\
\intertext{Now appealing to the Gehring-Hayman inequality we get }
|\sigma_{pq}(t)-1|&\leq C|p-q|+\delta_{\varepsilon}\notag \leq
C\big(|p-1|+|q-1|\big)+\delta_{\eps}
\leq (2C+1)\delta_{\varepsilon}<\eps.
\end{align*}

\smallskip

\noindent
Hence, $\gamma_{pq}(t)\in B(e_1,\varepsilon)\cap \bn
\qquad
\forall\, t\in [0,\kb(p,q)].$
\noindent
\end{proof}
\begin{remark}
 In view of \cite[Theorem~1]{KNT2024}, the proposition remains valid, with the unit ball replaced by any bounded strongly convex domain with $\mathcal{C}^{2,\alpha}$-smooth boundary where $\alpha \in (0,1)$.
  
\end{remark}
Before goinng the the proof lets recall the notion of the geodesic visibility property of a domain. Let $\Omega\Subset\mathbb{C}^{d}$ be a complete hyperbolic domain. We say that
$\Omega$ satisfies the \emph{geodesic visibility property} if, for every pair
of distinct points $\xi,\eta\in\partial\Omega$, there exist neighborhoods
$U_{\xi}$ and $U_{\eta}$ of $\xi$ and $\eta$, respectively, with
$\overline{U_{\xi}}\cap\overline{U_{\eta}}=\varnothing$, and a compact set
$K\Subset\Omega$ such that every Kobayashi geodesic segment joining a point
$x\in U_{\xi}\cap\Omega$ to a point $y\in U_{\eta}\cap\Omega$ intersects $K$. For the convenience of the reader, we recall the theorem stated in the Introduction and provide its proof in this section.

\begin{theorem}\label{T:local to global}
 Let $\Omega \subset \mathbb{C}^{d}$ be a bounded, complete hyperbolic domain satisfying the geodesic visibility property. Assume that, for every $\xi \in \partial \Omega$, there exists an open neighborhood $U_{\xi}$ of $\xi$ such that $\Omega' := \Omega \cap U_{\xi}$ is connected and has the following property: for any two maps $\gamma_{1},\gamma_{2}:[0,\infty)\to\Omega'$ and some $a,b>0$, \[ \gamma_j\big|_{[s,t]} \text{ is a }(1,ae^{-bs})\text{-quasi-geodesic for every }0\le s<t, \quad j=1,2, \] implies that there exist $T_{0}\in\mathbb{R}$ and $A(\gamma_{1}(0),\gamma_{2}(0),a),B(b)>0$ such that \[ K_{\Omega'} \bigl(\gamma_{1}(t),\gamma_{2}(t+T_{0})\bigr) \leq A e^{-Bt} \] for all sufficiently large $t$. Then, for any two geodesic rays \[ \sigma_{1},\sigma_{2}:[0,\infty)\to\Omega \] satisfying $\sigma_{1}(\infty)=\sigma_{2}(\infty),$ there exist $T\in\mathbb{R}$ and $B'>0$ such that \[ \limsup_{t\to\infty} \frac{\log K_{\Omega} \bigl(\sigma_{1}(t),\sigma_{2}(t+T)\bigr)}{t} \leq -B'. \]
 
\end{theorem}
\begin{proof}
Let $\sigma_j:[0,\infty)\longrightarrow \Omega, ~~~ j=1,2,$ be two geodesic rays such that $\sigma_1(\infty)=\sigma_2(\infty)=:\xi\in\partial\Omega.$ Let $U_{\xi}$ be the neighborhood of $\xi$ given by the hypothesis. Clearly, there exists $t_{0}>0$ such that $\sigma_{j}([t_{0}, \infty)) \st U_{\xi}$ for $j=1,2$.  We shall show that each $\sigma_j$, when regarded
as a curve in $\Omega'$, is a $(1,ae^{-bs})$-quasi-geodesic for suitable
constants $a,b>0$.  We first prove that there exists a constant $\lambda>0$ such that
such that
\begin{equation*}\label{eq:distance-complement}
 K_{\Omega}\bigl(\sigma_j(t),\Omega\setminus U_{\xi}\bigr)
 \geq t-\lambda \qquad \forall t>t_{0} \qquad j=1,2.
\end{equation*}
We prove this for $\sigma_1$; the argument for $\sigma_2$ is identical. Since $\Omega$ is complete hyperbolic, for every  $t>t_0$ there exists $x_{t}\in \partial U_{\xi} \cap \Om$  such that \begin{equation*}
K_{\Omega}\bigl(\sigma_1(t),\Omega\setminus U_{\xi}\bigr) = K_{\Omega}\bigl(\sigma_1(t),x_{t}\bigr). \end{equation*}

\smallskip

\noindent
We now claim the following:

\smallskip

\noindent
{\bf Claim: } There exists $\lambda>0$ independent of $t$ such that we can choose 
 $t_{0}>0$ be large enough such that 
\begin{equation}\label{eq:minimizer}
    K_{\Omega}\bigl(\sigma_1(t),x_t\bigr) \geq t-\lambda \quad \forall t\geq t_{0}.
\end{equation}
{\bf Proof of the claim:} Suppose first that there exist  a compact set $L\Subset\Omega$ such that $ x_t\in L $ for every $t\geq t_0$. Then, by the reverse triangle inequality, we conclude the following holds for all $t>t_{0}$.
\begin{align*} K_{\Omega}\bigl(\sigma_1(t),x_t\bigr) &\geq K_{\Omega}\bigl(\sigma_1(0),\sigma_1(t)\bigr) - K_{\Omega}\bigl(\sigma_1(0),x_t\bigr)\geq t-\lambda,
\end{align*} 
where $\lambda=\sup_{x\in L}K_{\Omega}\bigl(\sigma_1(0),x \bigr).$

\medskip

\noindent
It remains to consider the case where the minimizing points admits a subsequence which leaves every compact subset of $\Omega$. 
Since  $x_{t} \in \partial\Om \cap  \partial U_{\xi}$, and $\Omega$ is bounded,  we may assume that there exists a subsequence  $x_{t_{\nu}}$ such that  \[ x_{t_\nu}\to\eta \in\partial\Omega\cap\partial U_{\xi}.\] 
In particular, $\eta\neq\xi$.   
For each $\nu$, let \[ \alpha_{\nu}: [0,K_{\Omega}(\sigma_1(t_\nu),x_{t_\nu})] \longrightarrow\Omega \] be a $\Om$-geodesic segment joining $\sigma_1(t_\nu)$ to $x_{t_\nu}$.
It follows from the geodesic visibility property that, for every
$\eta\in\partial\Omega\cap\partial U_{\xi},$
there exists neighborhoods $U_{\eta}$ of $\eta$ and $U_{\xi}^{\eta}$ of
$\xi$, and a compact set $K_{\eta}\Subset\Omega$, such that every
$\Omega$-geodesic joining a point $x\in U_{\eta}\cap\Omega$ to a point  $y\in U_{\xi}^{\eta}\cap\Omega$
intersects $K_{\eta}$. Since $\partial\Omega\cap\partial U_{\xi}$ is compact, there exist finitely
many points
\[
\eta_1,\ldots,\eta_m\in\partial\Omega\cap\partial U_{\xi}
\]
such that
\[
\partial\Omega\cap\partial U_{\xi}
\subset
\bigcup_{\ell=1}^{m}U_{\eta_\ell}.
\]
Let
\[
K:=\bigcup_{\ell=1}^{m}K_{\eta_\ell}.
\]
Then $K\Subset\Omega$ is compact. Moreover, since the intersection
\[
V_{\xi}:=\bigcap_{\ell=1}^{m}U_{\xi}^{\eta_\ell}
\]
is a neighborhood of $\xi$, every $\Omega$-geodesic joining a point in
$V_{\xi}\cap\Omega$ to a point in $\bigcup_{\ell=1}^{m}(U_{\eta_\ell}\cap\Omega)$
 intersects the fixed compact set $K$. Since $\sigma_1(t)\to\xi$ as $t\to\infty$, we can assume that $t_{0}>0$ is large enough such that
 \[
\sigma_1(t)\in V_{\xi}\quad \forall t\geq t_{0}.
\]
Thus, whenever $x_t$ is sufficiently close
to $\partial\Omega\cap\partial U_{\xi}$, every $\Omega$-geodesic joining
$\sigma_1(t)$ to $x_t$ intersects the fixed compact set $K$. Therefore,

 \[ \alpha_{\nu}([0,K_{\Omega}(\sigma_1(t_\nu),x_{t_\nu})])\cap K\neq\varnothing, \] for all sufficiently large $\nu$. Choose \[ y_{\nu}\in \alpha_{\nu}\big([0,K_{\Omega}(\sigma_1(t_\nu),x_{t_\nu})]\big)\cap K. \] Since $y_{\nu}$ lies on the geodesic segment joining $\sigma_1(t_\nu)$ to $x_{t_\nu}$, we have \[ K_{\Omega}\bigl(\sigma_1(t_\nu),x_{t_\nu}\bigr) \geq K_{\Omega}\bigl(\sigma_1(t_\nu),y_{\nu}\bigr). \] Again, by the  triangle inequality, \begin{align*} K_{\Omega}\bigl(\sigma_1(t_\nu),y_{\nu}\bigr) &\geq K_{\Omega}\bigl(\sigma_1(0),\sigma_1(t_\nu)\bigr) - K_{\Omega}\bigl(\sigma_1(0),y_{\nu}\bigr)\geq t_\nu- \lambda_{K}, \end{align*}
where $\lambda_K:= \sup_{y\in K}K_{\Omega}\bigl(\sigma_1(0),y\bigr)$. Therefore, combining the above two relations we get
 \[ K_{\Omega}\bigl(\sigma_1(t_\nu),x_{t_\nu}\bigr) \geq t_\nu-\lambda_K. \] Thus, in either case, there exists a constant $\lambda>0$ such that \[ K_{\Omega}\bigl(\sigma_1(t),x_t\bigr) \geq t-\lambda \quad \forall t\geq t_{0}.\]  
In view of \eqref{eq:minimizer}, this gives \[ K_{\Omega}\bigl(\sigma_1(t),\Omega\setminus U_{\xi}\bigr) \geq t-\lambda  \quad \forall t\geq t_{0}. \] The same argument applies to $\sigma_2$, and hence, after increasing $\lambda$ if necessary, \[ K_{\Omega}\bigl(\sigma_j(t),\Omega\setminus\Omega'\bigr) \geq t-\lambda, \qquad j=1,2, \] for all $t\geq t_{0}$. This completes the proof of the claim. \hfill $\blacktriangleleft$

\smallskip

\noindent
 To this end let $j=1,2$ and  apply the Result~\ref{R:Roydenestimate} to obtain 
\[ \kappa_{\Omega'} \bigl(\sigma_j(t);\sigma_j'(t)\bigr) \leq \bigg(1+4e^{-K_{\Om}\big(\sigma_{j}(t); \Om\setminus U_{\xi}\big)}\bigg)\, \kappa_{\Omega} \bigl(\sigma_j(t);\sigma_j'(t)\bigr). \] Since $\sigma_j$ is parametrized by Kobayashi arc length, $\kappa_{\Omega} \bigl(\sigma_j(t);\sigma_j'(t)\bigr)=1 ~~{\rm a.e.} ~t \in [s,t_{2}]$ for any $t_{0}\leq s\leq t_{2}<\infty$. Therefore, assuming $t_{0}>0$ sufficiently large we conclude that there exists $a>0$ such that  
\begin{align*}
  \kappa_{\Omega'} \bigl(\sigma_j(t);\sigma_j'(t)\bigr) \leq 1+ae^{-t} \quad {\rm a.e.}~~ t \in [s, t_{2}]. 
\end{align*}
\noindent 
Now integrating the above equation from  $s$ to $t_{2}$ we get
 \begin{align*}
  \int_{s}^{t_{2}}\kappa_{\Omega'} \bigl(\sigma_j(t);\sigma_j'(t)\bigr) \leq \int_{s}^{t_{2}}1\,dt+a\int_{s}^{t_{2}}e^{-t} \,dt\notag\\
  \impl |s-t_{2}|=K_{\Om}(\sigma_{j}(s), \sigma_{j}(t_{2})) \leq  K_{\Om'}(\sigma_{j}(s), \sigma_{j}(t_{2}))\leq |t_{2}-s|+a e^{-s}.
 \end{align*}
Now  it follows from the  hypothesis  that there exists $T \in\re$ and $A, B>0$ such that 
$$K_{\Om'}(\sigma_{1}(t), \sigma_{2}(t+T))\leq Ae^{-Bt} \quad \forall t>\max\{t_{0}-T,0\}. $$ 

\smallskip
\noindent
Since $K_{\Om}(\sigma_{1}(t), \sigma_{2}(t+T))\leq K_{\Om'}(\sigma_{1}(t), \sigma_{2}(t+T))$, hence, we are done.

\end{proof}
\begin{theorem}\label{T:Glob to loc}
Let $\xi\in\partial\bn$ and let $U_{\xi}$ be an open neighbourhood of $\xi$ such that $B':=\bn\cap U_{\xi}$ is connected. Let $\gamma_1,\gamma_2:[0,\infty)\longrightarrow B'$ be two geodesic rays satisfying $\gamma_1(\infty)=\gamma_2(\infty)=\xi.$ Then there exist constants $a,b>0$ and $T\in\re$ such that
\begin{equation*}\label{eq:main-exp}
K_{B'}\bigl(\gamma_1(t),\gamma_2(t+T)\bigr)
    \le a e^{-bt}
    \qquad t> \max\{0, -T\}.
\end{equation*}
\end{theorem}
\begin{proof}
Let
\[
    \gamma_1,\gamma_2:[0,\infty)\longrightarrow B'
\]
be two geodesic rays with common endpoint $\xi$.
Choose $\varepsilon>0$ so small that
$B(\xi,\varepsilon)\cap\bn\subset B'.$ By localization of the Kobayashi distance (see \cite[Proposition~6.5]{BNT2022}), after decreasing $\varepsilon$
if necessary, we may assume that
\begin{equation}\label{eq:distance-localization-half}
    K_{\bn\cap B(\xi,\varepsilon)}(z,w)-K_{\bn}(z,w)<\frac12,
    \qquad z,w\in B(\xi,\delta_{\varepsilon})\cap\bn.
\end{equation}
We infer from   Lemma ~\ref{L:range of geodesic}, there exists
$\delta_{\varepsilon}>0$ such that, whenever
$p,q\in B(\xi,\delta_{\varepsilon})\cap\bn$, the image of the
$\bn$-geodesic segment $\sigma_{pq}$ joining $p$ and $q$ is contained in
$B(\xi,\varepsilon)\cap\bn$.
Since $\gamma_j(\infty)=\xi$, there exists $M_{\varepsilon}>0$ such that
\begin{equation*}\label{eq:eventually-small-ball}
    \gamma_j(t)\in B(\xi,\delta_{\varepsilon})\cap\bn,
    \qquad t>M_{\varepsilon},\quad j\in\{1,2\}.
\end{equation*}
Because each $\gamma_j$ is a $B'$-geodesic, for $s,t>M_{\varepsilon}$,
\begin{equation*}\label{eq:bp-geodesic}
    |t-s|=K_{B'}\bigl(\gamma_j(t),\gamma_j(s)\bigr)
    \ge K_{\bn}\bigl(\gamma_j(t),\gamma_j(s)\bigr).
\end{equation*}
\noindent
Combining the above relation and  \eqref{eq:distance-localization-half} we obtain
\begin{align}
K_{B'}\bigl(\gamma_j(t),\gamma_j(s)\bigr)-\frac12\leq K_{\bn\cap B(\xi,\varepsilon)}\bigl(\gamma_j(t),\gamma_j(s)\bigr)-\frac12
    &\le K_{\bn}\bigl(\gamma_j(t),\gamma_j(s)\bigr)\leq |t-s| \quad \forall s, t>M_{\eps}
\end{align}
\noindent
Thus for each $j \in \{1,2\}$, 
\begin{equation}\label{eq:first-quasigeodesic}
    |t-s|-\frac12
    \le K_{\bn}\bigl(\gamma_j(s),\gamma_j(t)\bigr)
    \le |t-s|,
    \qquad s,t>M_{\varepsilon}.
\end{equation}
\noindent
Hence, viewed inside $\bn$, each $\gamma_j$ is a $(1,\tfrac12)$-quasi-geodesic
on $[M_{\varepsilon},\infty)$. 

\medskip

\noindent
Fix $s,t>M_{\varepsilon}$. For $j\in\{1,2\}$, let
\begin{equation*}\label{eq:sigma-def}
    \sigma_j:[0,L_j]\longrightarrow\bn,
    \qquad
    L_j:=K_{\bn}\bigl(\gamma_j(s),\gamma_j(t)\bigr),
\end{equation*}
be the unit-speed $\bn$-geodesic joining $\gamma_j(s)$ to $\gamma_j(t)$.
By the choice of $\delta_{\varepsilon}$,
\begin{equation*}\label{eq:sigma-local-ball}
    \sigma_j([0,L_j])\subset B(\xi,\varepsilon)\cap\bn.
\end{equation*}
\noindent
By the Morse lemma (Result~\ref{R:Morselemma}) applied to the hyperbolic metric space $(\bn,K_{\bn})$,
there exists $R_1>0$, depending only on the hyperbolicity constant and the
quasi-geodesic constants, such that for every $\tau\in[0,L_j]$ there is a
point $a_j(\tau)\in[s,t]$ satisfying
\begin{equation}\label{eq:morse-close}
K_{\bn}\bigl(\gamma_j(a_j(\tau)),\sigma_j(\tau)\bigr)<R_1.
\end{equation}
\noindent
Note that $\sigma_j(0)=\gamma_j(s)$. Therefore, it turns out from \eqref{eq:morse-close} and \eqref{eq:first-quasigeodesic} that 
\begin{align}
    \tau &=K_{\bn}\bigl(\sigma_j(0),\sigma_j(\tau)\bigr)\le
      K_{\bn}\bigl(\gamma_j(s),\gamma_j(a_j(\tau))\bigr)
      +K_{\bn}\bigl(\gamma_j(a_j(\tau)),\sigma_j(\tau)\bigr) \le a_j(\tau)-s+R_1.
\end{align}
Therefore,
\begin{equation}\label{eq:aj-lower}
    a_j(\tau)\ge \tau+s-R_1.
\end{equation}
\noindent
It follows from the definition of distance to a set and
\eqref{eq:morse-close} that
\begin{equation}\label{eq:sigma-complement-reduction}
K_{\bn}\bigl(\sigma_j(\tau);\bn\setminus B(\xi, \eps)\bigr)
    \ge
K_{\bn}\bigl(\gamma_j(a_j(\tau));\bn\setminus B(\xi, \eps)\bigr)-R_1.
\end{equation}

\smallskip

\noindent
We next estimate the term on the right-hand side of the above equation. Let
\[
    \eta:[0,\infty)\longrightarrow\bn
\]
be the $\bn$-geodesic ray with $\eta(0)=0$ and    $\eta(\infty)=\xi$. 
Again, by applying the Morse lemma (Result~\ref{R:Morselemma}), we conclude that there exists a constant $C_1>0$ such that, for every $t\ge M_{\varepsilon}$, there exists $b_j(t)>0$ satisfying
\begin{equation}\label{eq:gamma-eta-close}
    K_{\bn}\bigl(\eta(b_j(t)),\gamma_j(t)\bigr)<C_1
    \qquad t\ge M_{\varepsilon}.
\end{equation}
\noindent
We now observe that  \eqref{eq:gamma-eta-close}, \eqref{eq:first-quasigeodesic} yields
\begin{align}\label{E:mm11}
    b_j(t)
&=K_{\bn}\bigl(\eta(0),\eta(b_j(t))\bigr) 
\le  K_{\bn}\bigl(\eta(0),\gamma_j(0)\bigr)+K_{\bn}\bigl(\gamma_j(0),\gamma_j(t)\bigr)+K_{\bn}\bigr(\gamma_j(t),\eta(b_j(t))\bigr) \notag\\
&\leq C_{1}+C+t,
    \end{align}
 where $C=\max\{K_{\bn}\bigl(\eta(0),\gamma_j(0)\bigr):j=1,2\}$. Similarly,  from \eqref{eq:first-quasigeodesic} and triangle inequality we get
\begin{align}\label{E:mm12}
    b_j(t)
    &\ge K_{\bn}\bigl(\eta(0),\gamma_j(t)\bigr)-C_1 \ge -K_{\bn}\bigl(\eta(0),\gamma_j(0)\bigr)
       +K_{\bn}\bigl(\gamma_j(0),\gamma_j(t)\bigr)-C_1 \notag\\
    &\ge t-\frac12-C_1- C.
\end{align}
Consequently, from \eqref{E:mm11} and \eqref{E:mm12}, it follows that  there exists $C_0>0$ such that for each $j=1,2$
\begin{equation}\label{eq:bj-time}
    |b_j(t)-t|\le C_0\quad \forall t>M_{\eps}.
\end{equation}

\smallskip

\noindent
Proceeding similar way as in the proof of  the  Theorem~\ref{T:local to global} (similar to the proof of the claim), we conclude that we can choose 
 $M_{\varepsilon}>0$ large enough  and $\lambda_0>0$ such that
\begin{equation}\label{eq:eta-complement}
K_{\bn}\bigl(\eta(u),\bn\setminus B(\xi, \eps)\bigr)
    \ge u-\lambda_0,
    \qquad u\ge M_{\varepsilon}.
\end{equation}
Using \eqref{eq:gamma-eta-close}, \eqref{eq:bj-time}, and
\eqref{eq:eta-complement}, we conclude there exists $M'_{\eps}>0$ such that 

\medskip

\begin{align}
K_{\bn}\bigl(\gamma_j(t);\bn\setminus  B(\xi, \eps)\bigr)
    &\ge
K_{\bn}\bigl(\eta(b_j(t));\bn\setminus B(\xi, \eps)\bigr)-C_1 
    \ge b_j(t)-\lambda_0-C_1 \ge t-\lambda'\quad \forall t>M'_{\eps},
\end{align}
where $\lambda'=C_{0}+\lambda_{0}+C_{1}$. From now we  choose $s,t>N_{\eps}:=\max\{M_{\eps}, M'_{\eps}\}$. Therefore, combining the above relation with \eqref{eq:aj-lower} and
\eqref{eq:sigma-complement-reduction}, we conclude that  there is a constant $\widetilde R>0$
such that
\begin{equation}\label{eq:sigma-far-complement}
    K_{\bn}\bigl(\sigma_j(\tau);\bn\setminus B(\xi, \eps)\bigr)
    \ge \tau+s-\widetilde R,
    \qquad 0\le \tau\le L_j,
\end{equation}

\noindent
We now apply Result~\ref{R:Roydenestimate} on the curve $\sigma_{j}$ to deduce the following:
\begin{equation}\label{eq:inf-localization}
    \kappa_{\bn\cap B(\xi,\varepsilon)}
      \bigl(\sigma_j(\tau);\sigma_j'(\tau)\bigr)
     \le
    \coth\!\left(
      K_{\bn}\bigl(\sigma_j(\tau);\bn\setminus  B(\xi,\varepsilon)\bigr)\kappa_{\bn}
      \bigl(\sigma_j(\tau);\sigma_j'(\tau)\bigr)
    \right).
\end{equation}
Here $\sigma_{j}$ is arc-length parametrized geodesic joining two points. Hence, $\kappa_{\bn}\big(\sigma_j(\tau);\sigma_j'(\tau)\big)=1$. From \eqref{eq:sigma-far-complement} and the elementary estimate
$\coth x\le 1+4e^{-2x}$ imply, there exists constant $A>0$, such that for large $s$,
\begin{equation}\label{eq:kappa-ratio-exp}
     \kappa_{\bn\cap B(\xi,\varepsilon)}
      \bigl(\sigma_j(\tau);\sigma_j'(\tau)\bigr)
     \le 1+A e^{-s}e^{-\tau}.
\end{equation}
\noindent
Hence, integrating the above equation with respect to $\tau$ from $0$ to $L_{j}$ we get 
\begin{align*}
  K_{\bn \cap B(\xi,\varepsilon)}(\gamma_{j}(s),\gamma_{j}(t)) \le L_j+A e^{-s}\int_0^{L_j}e^{-\tau}\,d\tau 
    =L_j+A e^{-s}(1-e^{-L_j})
    \le L_j+A e^{-s}.
\end{align*}
Note the the Kobayashi distance is non expanding under holomorphic maps. Consequently, we obtain that   \begin{equation}\label{eq:distance-upper-local}
  |s-t|- A e^{-s}= K_{B'}\bigl(\gamma_j(s),\gamma_j(t)\bigr)- A e^{-s}
    \le
    K_{\bn}\bigl(\gamma_j(s),\gamma_j(t)\bigr)\leq |s-t|, \quad \forall s,t \geq N_{\eps}.
\end{equation}
We now invoke Theorem~\ref{P:exponetial} to conclude that there exist $T_0\in\re$ and
$A_0>0$ such that
\begin{equation}\label{eq:ball-exp-convergence}
    K_{\bn}\bigl(\gamma_1(t),\gamma_2(t+T_0)\bigr)
    \le A_0 e^{-t/2}\quad \forall t>\max\{0,-T_{0}\}.
\end{equation}
Finally, assuming $\eps>0$ small enough, and $N_{\eps0>0}$ sufficiently large and applying the localization estimate on $\bn \cap B(\xi, \eps)$ (see \cite[Proposition~3]{Venturini1989}) we get 
\begin{equation}\label{eq:final-distance-ratio}
K_{\bn \cap B(\xi, \eps)}\bigl(\gamma_1(t),\gamma_2(t+T_0)\bigr)
\le \frac32 K_{\bn}\bigl(\gamma_1(t),\gamma_2(t+T_0)\bigr)\quad \forall t>N_{\eps}.
\end{equation}
\noindent
Now the relation $K_{B'}(z,w)\leq K_{\bn \cap B(\xi, \eps)}(z,w)$ together with \eqref{eq:ball-exp-convergence} completes the proof.

\end{proof}

\section{Appendix}\label{S:appendix}
Here we prove the following lemma.
\begin{lemma}\label{L:AppenLemma}
    Let $\triangle \bar{P}\bar{Q}\bar{R}$ and $\bar{\sigma}^{n}_{s}:[0, \rho_{\D}(\bar{P}, \bar{R})] \to \D$, be as Claim~\ref{claim:squaregromov}. Then there exists $A_{1}>0$ such 
    $$ \rho_{\D}\big(\bar{Q},\bar{\sigma}^{n}_{s}(a^{n}(s))\big)
\leq A_{1}
\sqrt{
\Big\langle
\bar{P}\mid \bar{R}
\Big\rangle_{\bar{Q}}^{\D}
}
\qquad \forall s\geq 0.$$
\end{lemma}

\begin{proof}
Let  the $\triangle{\bar{P}\bar{Q}\bar{R}}$ be as figure~\ref{fig:1} and 
$$x=\rho_{\D}\big(\bar{Q},\bar{R}\big);~~ y=\rho_{\D}\big(\bar{Q},\bar{P}\big);~~ z=\rho_{\D}\big(\bar{R},\bar{P}\big)=u+v; ~~r=\rho_{\D}\big(\bar{Q},\mathcal{\bar{P}}^n\big)$$ 
\noindent
where $v=\rho_{\D}(\bar{P}, \mathcal{\bar{P}}^n)$ and $u=\rho_{\D}(\bar{R}, \mathcal{\bar{P}}^n)$.  Since $\mathcal{\bar{P}}^n$ is the foot of the perpendicular from the point $\bar{Q}$ on to the side $[\bar{P},\bar{R}]$, we have 
$$\angle \bar{Q}\mathcal{\bar{P}}^n\bar{R}= \angle \bar{P}\mathcal{\bar{P}}^n\bar{Q}=\frac{\pi}{2}.$$
\noindent
Thus,  applying the hyperbolic law of cosine on 
$\triangle \bar{Q}\mathcal{\bar{P}}^n\bar{R}$ and $\triangle \bar{Q}\bar{P}\mathcal{\bar{P}}^n$ we obtain the following 
\begin{align*}
    \cosh{x}&=\cosh{r}\cosh{u}-\sinh{r}\sinh{u}.\cos{\angle \bar{Q}\mathcal{\bar{P}}^n\bar{R}}=\cosh{r}\cosh{u}\notag\\
    \cosh{y}&=\cosh{r}\cosh{v} .
\end{align*}
Now multiplying the above two equation we get that
\begin{align}\label{E:facte1}
    \cosh^2{r}&=\frac{\cosh{x}\cosh{y}}{\cosh{u}\cosh{v}}=e^{x+y-z}\Bigg[\dfrac{(1+e^{-2x})(1+e^{-2y})}{(1+e^{-2u})(1+e^{-2v})}\Bigg].
\end{align}
\noindent
From the hyperbolic sine law on the triangle 
$ \triangle\bar{Q}\bar{P}\mathcal{\bar{P}}^n$ we get  
\begin{align*}
    \frac{\sinh{y}}{1}&=\dfrac{\sinh{v}}{\sin{(\angle \mathcal{\bar{P}}^n\bar{Q}\bar{P}})}.
    \end{align*}
\smallskip

\noindent
Hence, $v \leq y$. Similarly,  from the  $ \triangle\bar{Q}\mathcal{\bar{P}}^n\bar{R}$ we deduce that $u\leq x$. Let $$\delta:=\Big\langle
\bar{P}\mid \bar{R}
\Big\rangle_{\bar{Q}}^{\D}.$$ Then,  by definition
$$2\delta=x+y-z>0.$$ Therefore, combining the above relations, we get that 
$$e^{-2u}-e^{-2x}\geq 0; \quad~~e^{-2v}-e^{-2y}\geq 0;\,\,\quad~~ e^{-2z}\geq e^{-2(x+y)}. $$
Consequently, we get that $\Bigg[\dfrac{(1+e^{-2x})(1+e^{-2y})}{(1+e^{-2u})(1+e^{-2v})}\Bigg]\leq1$. Therefore, it follows from \eqref{E:facte1} that 
\begin{align}
    \cosh^2{r}&\leq e^{2\delta}\notag\\
    e^{r}+e^{-r}&\leq 2e^{\delta}\notag\\
    r&\leq \ln{\big(e^{\delta}+\sqrt{e^{2\delta}-1}\big)}.\notag\\
    \intertext{Since we have noted in \eqref{E:ae7} that $\delta \to 0$ as $s \to \infty$, we get}
  r  &\leq \ln{\bigg(1+\delta +O(\delta^2)+\sqrt{\delta}O(1)\bigg)} {\rm as} ~~s \to \infty.
\end{align}
Note that $\frac{\ln{(1+x)}}{x}=O(1)$ as $x \to 0$ and  $(\delta +O(\delta^2)+\sqrt{\delta}O(1))=O(\sqrt{\delta})$ as $\delta \to 0$. Hence, we conclude that there exists $M_{0}, ~A_{1}>0$ such that $r\leq A_{1}\sqrt{\delta}$ for all $s>M_{0}$. Hence, we are done.

\end{proof}


\end{document}